\documentclass[10pt,a4paper]{amsart}
\usepackage[utf8]{inputenc}
\usepackage[T1]{fontenc}
\usepackage[french,english]{babel}
\usepackage{amsmath,amsfonts,amssymb,amsthm}
\usepackage{stmaryrd} 
\usepackage{array}
\usepackage{enumerate}
\usepackage{hyperref}
\newcommand{\Q}{\mathbb{Q}}
\newcommand{\N}{\mathbb{N}}
\newcommand{\Z}{\mathbb{Z}}
\newcommand{\R}{\mathbb{R}}
\newcommand{\F}{\mathbb{F}}
\newcommand{\C}{\mathbb{C}}
\newcommand{\Gal}{\mathrm{Gal}}
\newcommand{\ppcm}{\mathrm{ppcm}}
\newcommand{\pgcd}{\mathrm{pgcd}}
\newcommand{\h}{\mathrm{h}}

\theoremstyle{plain}
\newtheorem{thm}{Théorème}[section]
\newtheorem{prop}[thm]{Proposition}
\newtheorem{cor}[thm]{Corollaire}
\newtheorem{lmm}[thm]{Lemme}

\newtheorem{conj}[thm]{Conjecture}
\newtheorem{pb}[thm]{Problème}

\newtheorem{hyp}[thm]{Hypothèse}
\newtheorem*{lmm*}{Lemme}
\newtheorem*{thm*}{Théorème}
\newtheorem*{claim*}{Fait}

\theoremstyle{plain}
\newtheorem{rqu}[thm]{Remarque}

\theoremstyle{plain}
\newtheorem{defn}[thm]{Définition}
\newtheorem{ex}[thm]{Exemple}
\usepackage[all]{xy}
\usepackage{tikz}
\usetikzlibrary{shadows,decorations.pathmorphing}
\usepackage{comment}
\usepackage{ifthen}

\begin{document}

\selectlanguage{french}

\title[Minoration de la hauteur dans un compositum de corps de rayon]{Minoration de la hauteur de Weil dans un compositum de corps de rayon}
\date{}
\author{Arnaud Plessis}
\address{Universit\'e de Caen Normandie}
\email{arnaud.plessis@unicaen.fr}
\maketitle

\begin{abstract}
L'étude de la propriété $(B)$ a été motivée par la conjecture de Lehmer. 
Un corps algébrique a la propriété $(B)$ s'il existe une constante $c$ strictement positive telle que tout élément de ce corps d'ordre infini ait une hauteur supérieure à $c$.

Dans cet article, sous certaines conditions d'uniformité, on donnera un compositum d'une famille de corps ayant la propriété (B). 
\end{abstract}

\selectlanguage{english}

\begin{abstract}
The study of Property (B) starts as a special case of Lehmer's conjecture. An algebraic field is said to satisfy Property (B) if there exists a positive constant bounding by below the height of every point of infinite order.

In this paper we prove that, under certain uniformity conditions, the compositum of a familly of fields satisfies Property (B).
\end{abstract}

\selectlanguage{french}
 
\section{Introduction.} 

 Notons $\h$ la hauteur logarithmique et absolue de Weil (cf la section \ref{hauteur de Weil} pour les détails).
 Notons également $\mu_\infty$ l'ensemble des racines de l'unité de $\overline{\Q}$.
Une question importante concernant cette hauteur est la conjecture de Lehmer (cf \cite[ §13]{LehmerFactorizationofcertaincyclotomicfunctions}).    
  
  \begin{conj}[Lehmer]
  
  Il existe $c>0$ tel que pour tout $x\in\overline{\Q}^*\backslash\mu_\infty$, on ait: 
 \begin{equation*}
 \h(x)\geq\frac{c}{[\Q(x):\Q]}.
 \end{equation*}
  \end{conj}
 
  Actuellement, le meilleur résultat connu est dû à Dobrowolski : 
  
  \begin{thm}[Dobrowolski, \cite{OnaquestionofLehmerandthenumberofirreduciblefactorsofapolynomial}]
  Il existe $c>0$ tel que pour tout $x\in \overline{\Q}^*\backslash \mu_\infty$, 
  \begin{equation*}
 \h(x)\geq \frac{c}{D}\left(\frac{\log\log (3D)}{\log (2D)}\right)^3
  \end{equation*}
  où $D=[\Q(x) : \Q]$.
  \end{thm} 
  
  Afin de mieux comprendre cette conjecture, Bombieri et Zannier ont introduit la notion de propriété (B) (pour Bogolomov) ci-dessous. 
  
  \begin{defn}
  On dira qu'un corps $K\subset\overline{\Q}$ satisfait (ou a) la propriété (B) s'il existe un réel $T_0>0$ tel que l'ensemble
  \begin{equation*}
  \mathcal{A}(T_0)=\{\alpha\in K^*\vert\; \h(\alpha)\leq T_0\}
  \end{equation*}
  ne contienne que les racines de l'unité.
  \end{defn}
  
 Nous allons maintenant nous intéresser à certaines familles de corps qui ont la propriété (B). 
  D'après un théorème de Northcott, les corps de nombres ont la propriété (B).
  Ainsi, la notion de propriété (B) est intéressante uniquement lorsque l'on a des extensions infinies de $\Q$. 
  Citons quelques exemples de corps ayant la propriété (B).
   Le premier concerne les corps totalement réels (un corps $K$ est dit totalement réel si pour tout plongement $\sigma : K\hookrightarrow \C$, son image est dans $\R$). 
 
  \begin{thm}[Schinzel, \cite{SchinzelOntheproductoftheconjugatesoutsidetheunitcircleofanalgebraicnumber} et Smyth, \cite{Onthemeasureoftotallyrealalgebraicintegers}] \label{Schinzel}
  
  Soit $K$ un corps totalement réel. Alors pour tout $x\in K^*\backslash\mu_\infty$, on a:
\begin{equation*}
 \h(x)\geq \frac{1}{2} \log\left(\frac{1+\sqrt{5}}{2}\right).
\end{equation*}
  \end{thm}

Le prochain exemple est un analogue $p$-adique du théorème \ref{Schinzel}. 
On dit que $\alpha\in\overline{\Q}$ est un nombre totalement $p$-adique si $p$ est totalement décomposé dans $\Q(\alpha)$. 
Le corps $\Q^{tp}$ constitué de tous les nombres algébriques totalement $p$-adiques est alors une extension galoisienne de $\Q$. 

Soient $p$ un nombre premier et $K$ une extension algébrique de $\Q$. 
On dira que $K$ se plonge dans un corps local $M$ si, pour toute place $\nu$ de $K$ au-dessus de $p$, le complété de $K$ par rapport à $\nu$ est inclus dans $M$.
Par exemple, $\Q^{tp}$ se plonge dans $\Q_p$. 
Le prochain théorème montre que $\Q^{tp}$ a la propriété (B). 
Plus généralement:
  
\begin{thm}[Bombieri-Zannier, Theorem 2, \cite{BombieriZannierAnoteonheightsincertaininfiniteextensionsof}] \label{Thm 2 Bombieri-Zannier}
Soient $L/\Q$ une extension galoisienne et $p$ un nombre premier. Si $L$ peut être plongé dans une extension finie de $\Q_p$, alors $L$ a la propriété (B).
 \end{thm}
 
 Le prochain résultat montre que les extensions abéliennes de $\Q$ ont la propriété (B).
\begin{thm} [Amoroso-Dvornicich, \cite{MR1740514}] \label{Amoroso Dvornicich}

Pour tout $x\in (\Q^{ab})^*\backslash\mu_\infty$, on a: 
\begin{equation*} 
\h(x)\geq (\log 5)/12.
\end{equation*}

\end{thm} 
Dans \cite[Theorem 1]{HabeggerSmallheightandinfinitenonabelianextensions}, Habegger a montré l'analogue du théorème \ref{Amoroso Dvornicich} dans le cas des courbes elliptiques.
Plus précisément, si $E$ est une courbe elliptique définie sur $\Q$ sans multiplication complexe, alors $\Q(E_\mathrm{tors})$ a la propriété (B) où $E_\mathrm{tors}$ désigne l'ensemble des points de torsion de $E$.

Dans \cite[Theorem 1.2]{AmorosoZannierAuniformrelativeDobrowolskislowerboundoverabelianextensions}, les auteurs ont étendu le théorème \ref{Amoroso Dvornicich} en montrant que pour tout corps de nombres $K$ de degré $d$ sur $\Q$ et tout $x\in (K^{ab})^*\backslash\mu_\infty$, on a: 
\begin{equation*} 
\h(x)>3^{-d^2-2d-6}.
\end{equation*}
  On peut se demander si l'hypothèse "$K$ est un corps de nombres" peut être remplacée par "$K$ a la propriété (B)".
  Cela n'est pas possible en général.
   Dans \cite[Theorem 5.3]{amorosodavidzannieronfieldswiththepropertyB}, les auteurs donnent un contre-exemple en prenant pour $K$ le corps $\Q^{tr}$ (le compositum de toutes les extensions totalement réelles), qui possède la propriété (B) par le théorème \ref{Schinzel}, et montrent que $\Q^{tr}(i)\subset K^{ab}$ n'a pas la propriété (B).
   
  \begin{defn}
Soient $K\subset \overline{\Q}$ un corps et $\nu$ une place finie de $K$.
 On dit qu'une extension algébrique $L/K$ a un degré local borné en $\nu$ s'il existe un entier positif $B(\nu)$ tel que $[L_w : K_\nu]\leq B(\nu)$ pour toute place finie $w$ de $L$ étendant $\nu$.
\end{defn}

 On a vu que si $K$ est un corps ayant la propriété (B), alors $K^{ab}$ n'a pas forcément la propriété (B). 
 On peut alors se poser la même question en remplaçant "corps ayant la propriété (B)" par "corps ayant un degré local borné en un premier $p$". 
 Cela conduit au problème suivant dont une réponse partielle a été donnée dans \cite{amorosodavidzannieronfieldswiththepropertyB}:
 
\begin{pb}
Si une extension algébrique $L/\Q$ a un degré local borné en un premier $p$ fixé, est-il vrai que $L^{ab}$ a la propriété (B) ? Plus généralement, si $p$ est un nombre premier et $(L_i)_i$ une suite de corps algébriques ayant tous un degré local borné en $p$, est-il vrai que le compositum des $L_i^{ab}$ a la propriété (B) ?
\end{pb}

De manière générale, la seconde affirmation est fausse. 
Pour $p$ fixé, il suffit de prendre la suite $(L_i)_i$ de tous les corps de nombres.
Ainsi, le compositum des $L_i^{ab}$ est égal à $\overline{\Q}$ qui n'a pas la propriété $(B)$. 

Dans cet article, on imposera en plus que la suite $([L_i : \Q])_i$ est majorée. 
On étudiera donc le problème suivant.
  
\begin{pb}
Soit $(K_n)_{n}$ une suite de corps de nombres telle que la suite $([K_n:\Q])_n$ est majorée. 
 Est-il vrai que le compositum de tous les $K_n^{ab}$ a la propriété (B) ?   
\end{pb}

Pour un corps de nombres $K$, on note $H(K)$ le corps de Hilbert de $K$.
  Galateau a considéré un cas particulier de cette question. Il a montré que :

\begin{thm} [Galateau, §5, \cite{GalateauSmallheightinfieldsgeneratedbysingularmoduli}]\label{Gal(B)}
Soit $(K_n/\Q)_n$ une suite d'extensions finies galoisiennes. Si :

\begin{enumerate} [i)]
\item la suite $([K_n : \Q])_n$ est majorée;
\item il existe un premier $p$ inerte dans tous les $K_n$;
\item les discriminants des $K_n$ sont deux à deux premiers entre eux;
\end{enumerate}
alors le compositum $L$ de tous les $H(K_n)$ satisfait la propriété (B). 
\end{thm}   
  
  La condition $iii)$ est assez restrictive. 
  Par un théorème de Dirichlet (cf \cite[Chapter III, theorem 22]{FMJT}), cela implique que tout premier $q\in\Q$ doit être ramifié dans au plus un des $K_n$. 
Galateau a conjecturé qu'à l'aide d'une étude locale, on pouvait se passer de cette condition.

   Soit $K$ un corps de nombres ou une extension finie de $\Q_p$. 
   Notons $\mathcal{O}_K$ son anneau des entiers. 
   
   Notre résultat principal consiste à supprimer la condition $iii)$ du théorème \ref{Gal(B)} et à affaiblir la condition $ii)$ en supposant seulement qu'il existe un unique idéal premier de $\mathcal{O}_{K_n}$ au-dessus de $p$. 
   En particulier, on autorise $p$ à être ramifié dans $K_n$.
   
  \begin{thm} \label{AmelioGal}
Soit $(K_n/\Q)_n$ une suite d'extensions finies telle que la suite $([K_n : \Q])_n$ est majorée. Supposons qu'il existe un premier $p\in\Q$ tel que pour tout $n$, il existe un unique idéal premier $\mathfrak{p}_n$ de $\mathcal{O}_{K_n}$ au-dessus de $p$. Alors le compositum $L$ de tous les $H(K_n)$ satisfait la propriété (B). 
\end{thm}
 
 Par définition, $H(K_n)$ est l'extension abélienne maximale de $K_n$ non-ramifiée.
 Cependant, dans la preuve du théorème \ref{AmelioGal}, on utilisera uniquement le fait que pour tout $n$, $\mathfrak{p}_n$ ne se ramifie pas dans $H(K_n)$.
 Cela nous a donc amené à généraliser le théorème \ref{AmelioGal} en remplaçant le corps de Hilbert par les corps de rayons :
 
\begin{thm} \label{version quantitative du thm principal.}
Soient $p\in \Q$ un nombre premier, $K$ un corps de nombres et $\mathfrak{p}$ un idéal premier de $\mathcal{O}_K$ au-dessus de $p$.
Soit $(K_n/K)_{n\in\N}$ une suite d'extensions finies telle que la suite $([K_n : K])_n$ est majorée et telle que pour tout $n$, il existe un unique idéal premier $\mathfrak{p}_n$ de $\mathcal{O}_{K_n}$ au-dessus de $\mathfrak{p}$.
Fixons $M\geq 1$ un entier.
Pour $n\in\N$, notons 
\begin{equation*}
\Omega_n=\{\text{idéaux}\; \mathfrak{m}\subset \mathcal{O}_{K_n} \vert \; \exists m\leq M, \; p\nmid m, \; \text{et} \;\mathfrak{m}\; \text{au-dessus de m}\}.
\end{equation*}
Pour $\mathfrak{m}\in\Omega_n$, notons $K_{n,\mathfrak{m}}$ le corps de rayon de $K_n$ associé à $\mathfrak{m}$.
Alors le compositum $L$ de tous les $K_{n,\mathfrak{m}}$ a la propriété $(B)$.
\end{thm}  

\vspace{0.1cm}
 
Une fois cet article prépublié, Pottmeyer m'a signalé qu'il avait récemment montré (cf \cite[Theorem 2.3]{Smalltotally}) un cas particulier du théorème \ref{AmelioGal} en prenant pour $K_n$ le corps $\Q(\sqrt{-q_n})$ avec $q_n\in\Q$ un premier tel qu'il existe un premier $p$ inerte dans tous les $K_n$ (cf aussi l'exemple \ref{ex Galateau}). 

Sa méthode peut probablement se généraliser au cas des extensions $K_n/\Q$ de degrés quelconques, mais toujours sous les hypothèses que la suite $([K_n : \Q])_n$ est majorée, que $K_n/\Q$ est galoisienne et qu'il existe un premier $p$ inerte dans tous les $K_n$.

\section*{Remerciements}

Je tiens à remercier chaleureusement S. Checcoli ainsi que I. Del Corso pour les conversations instructives que l'on a eues et qui ont permis l'amélioration de cet article. 
Je remercie également L. Pottmeyer pour nos échanges pertinents qui ont été d'une grande aide.
Je remercie également F. Amoroso pour ses conseils tout au long de l'élaboration de cet article. 
Enfin, je remercie l'arbitre pour ses précieuses remarques aussi pertinentes qu'utiles. 

\section{Rappels et notations.}      
      
\subsection{Hauteur de Weil.} \label{hauteur de Weil}

 Soient $K$ un corps de nombres, $x\in K$ et $\mathfrak{p}$ un idéal premier de $\mathcal{O}_K$ au-dessus d'un premier $p\in \Q$. 
Notons $v_\mathfrak{p}(x)$ la valuation de $x$ en $\mathfrak{p}$ (i.e. l'exposant de $\mathfrak{p}$ dans la décomposition en idéaux premiers de $x\mathcal{O}_K$) et $\vert x\vert_\mathfrak{p}=p^{-v_\mathfrak{p}(x)/v_\mathfrak{p}(p)}$.
L'ensemble des places finies de $K$ est noté $\mathcal{M}^0(K)$ et celui des places infinies $\mathcal{M}^\infty(K)$. 
Notons également $\mathcal{M}(K)=\mathcal{M}^0(K) \cup \mathcal{M}^\infty(K)$.
Pour tout $\nu\in\mathcal{M}^0(K)$, on confond $\nu$ et son idéal premier associé. 
De même, pour tout $\nu\in \mathcal{M}^\infty(K)$, on confond $\nu$ et ses plongements associés. 
 
 Soit $\nu$ une place de $K$. 
 On note $K_\nu$ le complété de $K$ en $\nu$. 
 Si $\nu\in\mathcal{M}^0(K)$, on note $d_\nu$ le degré local $[K_\nu:\Q_p]$ où $p$ est le premier au-dessous de $\nu$. 
De la même façon, si $\nu\in\mathcal{M}^\infty(K)$, on note $d_\nu=[K_\nu : \R]$.
 
 Avec ces différentes notations, on a :
 
 \begin{prop}[Formule du produit, Proposition 1.4.4, \cite{BombieriandGublerHeightsinDiophantineGeometry}]
  Pour tout $x\in K^*$, 
  \begin{equation*}  
 \sum_{\nu\in\mathcal{M}(K)}d_\nu\log\vert x\vert_\nu=0.
  \end{equation*}
 \end{prop}  

Définissons maintenant la hauteur logarithmique et absolue de Weil.

\begin{defn}
Soient $x\in\overline{\Q}^*$ et $K$ un corps de nombres le contenant. On définit la hauteur logarithmique et absolue de Weil, notée $\h(x)$, par: 
\begin{equation*}
\h(x)=\frac{1}{[K:\Q]}\sum_{\nu\in\mathcal{M}(K)}d_\nu\log\max\{1,\vert x\vert_\nu\}.
\end{equation*}
\end{defn}

\begin{rqu}
Cette hauteur est appelée absolue car elle ne dépend pas du corps contenant $x$ (cf \cite[§1.5]{BombieriandGublerHeightsinDiophantineGeometry}).
\end{rqu}

Pour simplifier, on désigne par le mot "hauteur" la hauteur logarithmique et absolue de Weil. 
Cette hauteur possède d'autres propriétés importantes comme celles ci-dessous (cf \cite[Proposition 1.5.17, Lemma 1.5.18]{BombieriandGublerHeightsinDiophantineGeometry}):

\begin{prop}
Soient $x\in \overline{\Q}$ et $\lambda\in\Z$. 
Alors $\h(x^\lambda)=\vert\lambda\vert \h(x)$ et, pour tout $\sigma \in \Gal(\overline{\Q}/\Q)$, on a $\h(\sigma(x))=\h(x)$.
\end{prop}

La hauteur est aussi utilisée pour prouver la finitude d'un ensemble de points algébriques via le théorème fondamental ci-dessous :
  
  \begin{thm}[Northcott, §1.6, \cite{BombieriandGublerHeightsinDiophantineGeometry}] \label{Northcott}
  Fixons $B,D>0$. Alors l'ensemble $$\{\alpha\in\overline{\Q}^*\vert\; \h(\alpha)\leq B~\text{et}~[\Q(\alpha):\Q]\leq D\}$$ est fini. 
  \end{thm}

Soit $\alpha\in\overline{\Q}^*$ tel que $\h(\alpha)=0$. Alors, pour tout entier $k\geq 1$, 
\begin{equation*}
\h(\alpha^k)=k\;\h(\alpha)=0.
\end{equation*}
 Par le théorème \ref{Northcott}, la suite $(\alpha^k)_{k\geq 1}$ ne possède alors qu'un nombre fini de termes distincts. 
 Ainsi, $\alpha\in \mu_\infty$. La réciproque étant triviale, on en déduit que : 
	
	\begin{prop}
	Soit $\alpha\in\overline{\Q}^*$. Alors, $\h(\alpha)=0$ si et seulement si $\alpha\in\mu_\infty$.
	\end{prop}
       
      \subsection{Corps de rayons.}
      
     Dans toute cette sous-section, $K$ désigne un corps de nombres. Nous allons rappeler quelques propriétés utiles sur les corps de classes de rayons. 
      
     \begin{defn}
      Un $K$- module est un produit formel $\mathfrak{m}=\mathfrak{m}_f\infty$ où $\infty$ désigne le produit formel de tous les plongements réels de $K$ et $\mathfrak{m}_f$ (appelée la partie finie) est un idéal non-nul de $\mathcal{O}_K$.
      \end{defn}
      
\`A partir de maintenant, $\mathfrak{m}=\mathfrak{m}_f\infty$ désigne un $K$-module et on confondra souvent $\mathfrak{m}$ et $\mathfrak{m}_f$. De ce module, on peut construire deux groupes importants. Le premier, est le groupe abélien $J_\mathfrak{m}(K)$ composé de tous les idéaux fractionnaires de $K$ premiers à $\mathfrak{m}_f$ et le second, noté $P_\mathfrak{m}(K)$, est le sous-groupe des idéaux principaux de $J_\mathfrak{m}(K)$ engendrés par $a/b$ où 
      
\begin{enumerate} [i)]
\item $a,b\in \mathcal{O}_K\backslash\{0\}$ et $\pgcd((a),(b))=1$, 
\item $a\equiv b \mod\mathfrak{m}$, 
\item $\sigma\left(\frac{a}{b}\right)>0$ pour tout plongement réel $\sigma$ de $K$.
\end{enumerate}
         
Un groupe $H$ est appelé \textit{groupe idéal de module $\mathfrak{m}$} si  $P_\mathfrak{m}(K)\subset H \subset J_\mathfrak{m}(K)$. Notons $Cl_\mathfrak{m}(K,H)$ le groupe quotient $J_\mathfrak{m}(K)/H$. Alors: 

\begin{thm} (\cite[chapter 6, Proposition 1.8]{NeukirchAlgebraicnumbertheory}) \label{généralisation groupe de classe}
Pour tout groupe idéal $H$ de module $\mathfrak{m}$, $Cl_\mathfrak{m}(K,H)$ est fini.
\end{thm}     

En particulier, le théorème \ref{généralisation groupe de classe} appliqué à $\mathfrak{m}_f=(1)$ et à $H=P_{(1)}(K)$  montre que le groupe des classes classique est fini.

Soient $K$ un corps de nombres, $L/K$ une extension finie et $\mathfrak{m}$ un $K$-module.
Notons 
\begin{equation*}
N_\mathfrak{m}(L/K)=\{\mathfrak{a}\subset K \vert \mathfrak{a}=N_{L/K}(\mathcal{A}),\; \mathcal{A} \; \text{un idéal fractionnaire de}\; L, \; \pgcd(\mathfrak{a},\mathfrak{m}_f)=1\}
\end{equation*}
et
\begin{equation*}
H_\mathfrak{m}(L/K)=P_\mathfrak{m}(K) N_m(L/K).
\end{equation*}
Si $L/K$ est galoisienne alors, par un théorème de Weber (cf \cite{MR1037805}),
\begin{equation} \label{inégalité sur les corps de rayons}
[J_\mathfrak{m}(K) : H_\mathfrak{m}(L/K)] \leq [L : K].
\end{equation}
On dit que $L$ est un \textit{corps de classes} (au sens de Takagi) sur $K$ s'il existe un $K$-module $\mathfrak{m}$ tel que \eqref{inégalité sur les corps de rayons} soit une égalité.
Un tel $K$-module est appelé un \textit{module admissible} pour $L/K$.
On dit également qu'un groupe idéal $H$ (associé au $K$-module $\mathfrak{m}$) admet un corps de classes $L$ si $H=H_\mathfrak{m}(L/K)$ et si \eqref{inégalité sur les corps de rayons} est une égalité.

\begin{thm}[Takagi] \label{Takagi}
Soit $H$ un groupe idéal d'un $K$-module $\mathfrak{m}$. Alors:
\begin{enumerate} [i)]
\item (Existence) $H$ admet un corps de classes $L$. \
\item (Isomorphisme) si $H$ admet un corps de classes $L$, alors $\Gal(L/K)\simeq Cl_\mathfrak{m}(K,H)$.\
\item (Complétude) toute extension abélienne de $K$ est un corps de classes sur $K$.\
\item (Comparaison) si $H_1$ (resp. $H_2$) est un groupe idéal associé au $K$-module $\mathfrak{m}$ et admet un corps de classes $L_1$ (resp. $L_2$), alors 
\begin{equation*}
L_1\subset L_2\Leftrightarrow H_2\subset H_1. \
\end{equation*}
\item (Décomposition) si $H$ admet un corps de classes $L$, alors tout premier $\mathfrak{p}$ de $\mathcal{O}_K$ ne divisant pas $\mathfrak{m}$ est non ramifié dans $L$ et son degré résiduel $f_\mathfrak{p}(L\vert K)$ est égal à l'ordre de $\mathfrak{p}$ dans $J_\mathfrak{m}(K)/H$.

\end{enumerate}
\end{thm}

Soient $\mathfrak{m}$ un $K$- module et $H$ un groupe idéal de module $\mathfrak{m}$.
Le $i)$ et le $iv)$ du théorème \ref{Takagi} prouvent respectivement l'existence et l'unicité du corps de classes pour $H$.
Si $H=P_\mathfrak{m}(K)$, le corps de classes pour $H$ est appelé corps de rayon et est noté $K_\mathfrak{m}$. 

Soit $L/K$ une extension abélienne finie. 
Alors, par le $iii)$ du théorème \ref{Takagi}, on peut trouver un $K$-module $\mathfrak{m_1}$ et un groupe idéal $H_1$ de module $\mathfrak{m_1}$ tels que $L$ est un corps de classes pour $H_1$. 
Or, $P_\mathfrak{m_1}(K)\subset H_1$ et donc $L\subset K_{\mathfrak{m_1}}$ d'après le $iv)$ du théorème \ref{Takagi}. 
 Ainsi, les corps de rayons jouent le rôle des extensions cyclotomiques dans le sens où toute extension abélienne finie de $K$ est contenue dans un corps de rayon de $K$.

Le cas $\mathfrak{m}=(1)$ fournit des résultats déjà connus sur les corps de Hilbert. 
Dans ce cas, le corps de rayon est le corps de Hilbert du corps $K$.

\subsection{Majoration du degré d'inertie.} \label{majoration du degré d'inertie}
 
 On s'intéresse maintenant à la majoration du degré d'inertie dans un compositum de corps.  
 Jusqu'à la fin de cette sous-section, fixons un premier rationnel $p$ et un corps local $F/\Q_p$.
Rappelons d'abord un résultat de Krasner qui donne le nombre d'extensions et d'extensions totalement ramifiées de $F$ de degré fixé.

\begin{thm}[Krasner, Théorème 2, \cite{KrasnerNombredesextensionsdundegredonneduncorpspadique}] \label{Krasner}
Notons $\mathcal{N}_{F,d}$ le nombre d'extensions $K/F$ de degré $d=hp^m$ avec $(h,p)=1$ et $D=d[F:\Q_p]$. Alors
\begin{equation*}
\mathcal{N}_{F,d}=\left(\sum_{l\mid h} l\right)\left(\sum_{s=0}^m \frac{p^{m+s+1}-p^{2s}}{p-1}(p^{\epsilon(s)D}-p^{\epsilon(s-1)D})\right),
\end{equation*}
où $\epsilon(s)=p^{-1}+p^{-2}+...+p^{-s}$ si $s>0$, $\epsilon(0)=0$ et $p^{\epsilon(s-1)D}=0$ si $s=0$.

De plus, si $\mathcal{N}_{F,d}^{(r)}$ désigne le nombre d'extensions totalement ramifiées $K/F$ de degré $d$, alors
\begin{equation*}
\mathcal{N}_{F,d}^{(r)}=d\sum_{s=0}^m p^s(p^{\epsilon(s)D}-p^{\epsilon(s-1)D}).
\end{equation*}
\end{thm} 

   Le théorème de Krasner donne en particulier une majoration de l'indice de ramification et du degré d'inertie du compositum d'une famille d'extensions de $F$.

\begin{cor} \label{corollaire immédiat du thm de Krasner}
Soit $(K_i/F)_{i\in\N}$ une famille d'extensions finies telle que la suite $([K_i : F])_i$ est majorée. 
Notons $\Gamma=\{[K_i : F], i\in\N\}$ et $K$ le compositum des $K_i$.
Alors
\begin{equation*}
e(K\vert F),f(K\vert F)\leq e(K\vert F) f(K \vert F)=[K : F] \leq \prod_{d\in\Gamma} d^{\mathcal{N}_{F,d}}.
\end{equation*}
\end{cor}

Nous nous proposons de donner une majoration plus fine qui sera prouvée dans l'appendice.
Par le théorème \ref{Krasner}, la famille $(K_i/F)_{i\in\N}$ qui intervient dans le corollaire \ref{corollaire immédiat du thm de Krasner} est une famille finie; on la note maintenant $\{K_1/F, \dots, K_n/F\}$.
Quitte à permuter les corps $K_1,\dots,K_n$, on suppose que $K_1/F,..., K_m/F$ sont non sauvagement ramifiés et que $K_{m+1}/F,...,K_n/F$ sont sauvagement ramifiés.
Notons $e_i=e(K_i\vert F)$ et $f_i=f(K_i\vert F)$.
Pour tout $r\in\{1,\dots,n\}$, posons
\begin{equation} \label{définition de Lambda r}
\Lambda_r =\{e_1,\dots,e_r\}. 
\end{equation}
Soit $e\in\Lambda_n$. Notons $\mathcal{N}(e)$ le nombre d'extensions $K_i/F$ d'indice de ramification  $e$.
Enfin, pour tout premier $q$, on note 
\begin{equation} \label{définition de a(q) r} 
a_r(q):=\left(\sum\limits_{e\in\Lambda_r} v_q(e)\right)-\underset{e\in\Lambda_r}{\max}(v_q(e)).
\end{equation}

\begin{thm} \label{inertiecompositum}
Si $m=n$, alors $e(K_1\dots K_n\vert F)=\ppcm(e_1,\dots,e_n)$. 
Si $m<n$, alors 
\begin{equation*}
e(K_1\dots K_n\vert F)\leq \ppcm(e_1,\dots,e_{m+1})  e_{m+1}^{\mathcal{N}(e_{m+1})-1}\prod_{e\in\Lambda_n\backslash\Lambda_{m+1}} e^{\mathcal{N}(e)}.
\end{equation*}
On a également
\begin{equation*}
f(K_1...K_n\vert F)\leq\ppcm(f_1,...,f_n) E
\end{equation*}
où
\begin{equation} \label{expression de E sans sauvage}
E=\prod\limits_{e\in\Lambda_n} e^{\mathcal{N}(e)-1}\prod\limits_{q\in\mathcal{P}} q^{a_n(q)}
\end{equation}
si $m\geq n-2$ et
\begin{equation} \label{expression de E avec sauvage}
E = \prod\limits_{e\in\Lambda_{m+2}}e^{-1}\prod\limits_{q\in\mathcal{P}} q^{a_{m+2}(q)}\prod\limits_{e\in\Lambda_n} e^{\mathcal{N}(e)} 
\end{equation}
si $m< n-2$. 
\end{thm}
La valeur de $e(K_1\dots K_n\vert F)$ dans le cas où $m=n$ dans le théorème ci-dessus découle immédiatement du lemme d'Abhyankar ci-dessous : 

\begin{thm}[Lemme d'Abhyankar, Theorem 3, \cite{CornellOntheconstructionofrelativegenusfields}] \label{Abhyankar}
Soit $F/\Q_p$ une extension finie.
Soient $L_1/F$ et $L_2/F$ deux extensions finies telles que $p\nmid e(L_1\vert F)$ ou $p\nmid e(L_2\vert F)$. Alors 
\begin{equation*}
e(L_1L_2\vert F)=\ppcm(e(L_1\vert F), e(L_2\vert F)).
\end{equation*}
\end{thm}

 Il n'est pas toujours aisé de savoir si deux extensions de $F$ sont égales.
  Le nombre $\mathcal{N}(e)$ n'est donc pas toujours facile à calculer. 
  Si l'on veut une majoration plus explicite du degré résiduel $f(K_1\dots K_n\vert F)$, on peut majorer $\mathcal{N}(e)$ comme suit: soient $e\in\Lambda_n$ et $f\in\N^*$ et cherchons le nombre d'extensions $L/F$ telles que
\[
\begin{cases}
       e(L\vert F)=e  \\
        f(L\vert F)=f
    \end{cases}.
\]
Notons $F\{f\}$ l'unique extension non ramifiée de $F$ de degré $f$ (cf \cite[Chapter III, Theorem 25]{FMJT}).
Alors, le nombre d'extensions de $F$ de degré $ef$ et d'indice de ramification $e$ est égal au nombre d'extensions totalement ramifiées de $F\{f\}$ de degré $e$, i.e. à $\mathcal{N}_{F\{f\},e}^{(r)}$ qui se calcule à l'aide du théorème \ref{Krasner}. 
On en déduit donc que 
\begin{equation} \label{majoration du nombre N(e).}
\mathcal{N}(e)\leq\sum_{f\in\Lambda(e)}\mathcal{N}_{F\{f\},e}^{(r)},
\end{equation}
où $\Lambda(e)=\{f\in\N^*\vert\;\exists i\in\{1,\dots,n\},\;e=e_i \; \text{et} \; f=f_i\}$. \\

\begin{rqu}
Pour d'autres résultats concernant l'indice de ramification dans un compositum d'extensions sauvages bien particulières, le lecteur intéressé pourra voir \cite{MR3496675} ou \cite{MR2309533}.
\end{rqu}

\section{Preuve du théorème \ref{version quantitative du thm principal.} et exemples.} \label{section preuve du theoreme fondamental}

Dans cette section, nous allons démontrer le théorème \ref{version quantitative du thm principal.} et on en donnera quelques applications.
La démonstration repose sur une méthode désormais classique et qui sera utilisée dans le lemme ci-dessous.

\subsection{Preuve du théorème \ref{version quantitative du thm principal.} et minoration explicite de la hauteur.} \label{une minoration explicite}

Soient $K$ un corps de nombres, $L/K$ une extension algébrique et $\mathfrak{p}$ un idéal premier de $\mathcal{O}_K$ au-dessus d'un premier rationnel $p$. 

\begin{defn} \label{définitions de k }
Pour $e\in\N^*$, on définit les réels $\lambda(e,\mathfrak{p})$ et $\beta(e,\mathfrak{p})$ de la façon suivante.
Soit $k$ l'unique entier tel que \[p^{k-1}(p-1)\leq e< p^k(p-1),\] et notons, pour $\lambda\in\N$, $\beta_{\lambda}(e,p)=p^{\min\{\lambda,k\}}/e+\max\{0, \lambda-k\}$. 
Alors on note $\lambda(e,\mathfrak{p})$ le plus petit entier positif tel que 
\begin{equation*}
\beta_{\lambda(e,\mathfrak{p})}(e,p)\; [K_\mathfrak{p} : \Q_p] \log p> [K : \Q]\log 2, 
\end{equation*}
et on note $\beta(e,\mathfrak{p}):=\beta_{\lambda(e,\mathfrak{p})}(e,p)$.
\end{defn}

\begin{lmm} \label{lemme classique}
Soit $\nu\mid\mathfrak{p}$ une place finie de $\mathcal{O}_L$.
 Supposons que les familles de réels $(f_\nu(L\vert \Q))_{\nu\vert \mathfrak{p}}$ et $(e_\nu(L\vert \Q))_{\nu\vert \mathfrak{p}}$ soient majorées respectivement par $f$ et $e$.
Alors $L$ a la propriété $(B)$. 
De plus, pour tout $x\in L\backslash\mu_\infty$, 
\begin{equation*}
\h(x)\geq \frac{1}{p^{f+\lambda}+p^\lambda}\left(\frac{\beta[K_\mathfrak{p} : \Q_p]}{[K : \Q]}\log p-\log 2\right)>0
\end{equation*}
où $\lambda=\lambda(e,\mathfrak{p})$ et $\beta=\beta(e,\mathfrak{p})$.
\end{lmm}

\begin{proof}

Soit $M\subset L$ un corps de nombres contenant $K$.
Soit $w$ une place de $\mathcal{O}_M$ au-dessus de $\mathfrak{p}$. 
Notons $f_w=f_w(M\vert \Q)$ et $e_w=e_w(M\vert \Q)$.
Comme $\mathcal{O}_M/w\mathcal{O}_M\simeq\mathcal{O}_{M_w}/w\mathcal{O}_{M_w}$ et $e_w\leq e$, on déduit du petit théorème de Fermat que 
\begin{equation*}
\vert x^{p^{f_w}}-x\vert_w\leq p^{-1/e}
\end{equation*}
 pour tout $x\in\mathcal{O}_{M_w}$.
D'après \cite[Lemma 2.1]{amorosodavidzannieronfieldswiththepropertyB} \footnote{Une imprécision s'est glissée dans ce lemme. La dernière formule de l'énoncé doit s'écrire  \[s_{p,\rho}(\lambda)=p^{\min\{\lambda,k\}} \rho+ \max\{0,\lambda-k\}\] } avec $\rho=1/e$, on a
\begin{equation*}
\vert x^{p^{f_w+\lambda}}-x^{p^{\lambda}}\vert_w\leq p^{-\beta}.
\end{equation*}
Pour avoir une majoration sur tout $M$, on utilise un argument d'Habegger (cf \cite[Lemma 4.2]{HabeggerSmallheightandinfinitenonabelianextensions}). 
Soit $x\in M_w$ tel que $x\in\mathcal{O}_{M_w}$. Alors
\begin{equation*}
\vert x^{p^{f_w+\lambda}}-x^{p^\lambda}\vert_w\leq p^{-\beta}.
\end{equation*}

Soit $x\in M_w$ tel que $x\notin\mathcal{O}_{M_w}$. Alors $x^{-1}\in\mathcal{O}_{M_w}$ car $\mathcal{O}_{M_w}$ est un anneau de valuation. 
Ainsi, $\vert x^{-p^{f_w+\lambda}}-x^{-p^\lambda}\vert_w\leq p^{-\beta}$ ou encore 
\begin{equation*}
\vert x^{p^{f_w+\lambda}}-x^{p^\lambda}\vert_w\leq p^{-\beta}\vert x\vert_w^{p^{f_w+\lambda}+p^\lambda}.
\end{equation*}
Il s'ensuit donc que pour tout $x\in M_w$, et par conséquent pour tout $x\in M$, que 
\begin{equation} \label{majoration pas bonne}
\begin{aligned}
\vert x^{p^{f_w+\lambda}}-x^{p^\lambda}\vert_w & \leq p^{-\beta}\max\{1,\vert x\vert_w\}^{p^{f_w+\lambda}+p^\lambda} \\
& \leq p^{-\beta}\max\{1,\vert x\vert_w\}^{p^{f+\lambda}+p^\lambda}
\end{aligned}
\end{equation}
car $f_w\leq f$. 

Soit $x\in M^*\backslash\mu_\infty$. 
Alors $x^{p^{f_w+\lambda}}-x^{p^\lambda}\neq 0$ et, par la formule du produit, on obtient 
\begin{align*}
0 & =  \sum\limits_{w\in \mathcal{M}(M)}\frac{d_w}{[M:\Q]}\log\vert x^{p^{f_w+\lambda}}-x^{p^\lambda}\vert_w \\ 
 & =  \sum\limits_{w\vert \mathfrak{p}}\frac{d_w}{[M:\Q]}\log\vert x^{p^{f_w+\lambda}}-x^{p^\lambda}\vert_w+\sum\limits_{w\nmid \mathfrak{p}}\frac{d_w}{[M:\Q]}\log\vert x^{p^{f_w+\lambda}}-x^{p^\lambda}\vert_w.
 \end{align*}
En utilisant la majoration \eqref{majoration pas bonne} et la formule (cf \cite[Corollaire  1.3.2]{BombieriandGublerHeightsinDiophantineGeometry})  
\begin{equation*}
\sum_{w\vert\mathfrak{p}} d_w=[M: K][K_\mathfrak{p} : \Q_p]
\end{equation*}
 sur les places divisant $\mathfrak{p}$, on obtient que
\begin{equation} \label{fddsdgvfgfgfdfdds}
\begin{aligned}
0\leq -\beta\frac{[K_\mathfrak{p} : \Q_p]}{[K : \Q]}\log p & +(p^{f+\lambda}+p^\lambda)\sum_{w\mid \mathfrak{p}} \frac{d_w}{[M : \Q]}\log\max\{1, \vert x\vert_w\}\\
& +\sum\limits_{w\nmid \mathfrak{p}}\frac{d_w}{[M:\Q]}\log\vert x^{p^{f_w+\lambda}}-x^{p^\lambda}\vert_w.
\end{aligned}
\end{equation}
En utilisant les inégalités ultramétriques et triangulaires ainsi que la formule
\begin{equation*}
\forall ~ a,b\in\R^{+*}, ~ \log(a+b)\leq \log\max\{1,a\}+\log\max\{1,b\}+\log 2,
\end{equation*}
il en résulte que 
\begin{multline*}
\sum_{w\nmid \mathfrak{p}}\frac{d_w}{[M:\Q]}\log\vert x^{p^{f_w+\lambda}}-x^{p^\lambda}\vert_w \leq \sum\limits_{\underset{w\in\mathcal{M}^0(M)}{w\nmid \mathfrak{p},}}\frac{d_w}{[M:\Q]}\log\max\{\vert x^{p^{f_w+\lambda}}\vert_w, \vert x^{p^\lambda}\vert_w\} \\
+ \sum\limits_{\sigma\in\mathcal{M}^\infty(M)}\frac{d_\sigma}{[M:\Q]}\big(\log\max\{1,\vert x^{p^{f_w+\lambda}}\vert_w\}+\log\max\{1, \vert x^{p^\lambda}\vert_w\} +\log 2\big).
\end{multline*}
En injectant cette inégalité dans \eqref{fddsdgvfgfgfdfdds} et en utilisant le fait que $f_w\leq f$, on en déduit que
\begin{equation*}
0\leq (p^{f+\lambda}+p^\lambda)\h(x)-\beta\frac{[K_\mathfrak{p} : \Q_p]}{[K : \Q]}\log p+\log 2
\end{equation*}
 et donc que 
\begin{equation*}
\h(x)\geq \frac{1}{p^{f+\lambda}+p^\lambda}\left(\frac{\beta[K_\mathfrak{p} : \Q_p]}{[K : \Q]}\log p-\log 2\right)>0
\end{equation*}
par définition de $\beta$.

\end{proof}

Nous sommes maintenant en mesure de prouver notre théorème : \\

\textit{Démonstration du théorème \ref{version quantitative du thm principal.} :}

Fixons $K$ un corps de nombres et $(K_n/K)_n$ une suite d'extensions finies telle que la suite $(d_n)_n$ avec $d_n=[K_n : K]$ est majorée. 
Fixons également un idéal premier $\mathfrak{p}$ de $\mathcal{O}_K$ au-dessus d'un premier rationnel $p$.
Fixons $M\geq 1$ un entier.
Pour $n\in\N$, notons 
\begin{equation*}
\Omega_n=\{\text{idéaux}\; \mathfrak{m}\subset \mathcal{O}_{K_n} \vert \; \exists m\leq M, \; p\nmid m, \; \text{et} \;\mathfrak{m}\; \text{au-dessus de} \; m\}.
\end{equation*}
Pour $\mathfrak{m}\in\Omega_n$, notons $K_{n,\mathfrak{m}}$ le corps de rayon de $K_n$ associé à $\mathfrak{m}$.

On suppose que pour tout $n$, il existe un unique idéal premier $\mathfrak{p}_n$ de $\mathcal{O}_{K_n}$ au-dessus de $\mathfrak{p}$.
On se propose de montrer que le compositum $L$ de tous les $K_{n,\mathfrak{m}}$ a la propriété $(B)$.
Comme $K$ est un corps de nombres fixé, on déduit du lemme \ref{lemme classique} qu'il suffit de montrer que les familles de réels $(e_\nu(L\vert K))_{\nu\vert \mathfrak{p}}$ et $(f_\nu(L\vert K))_{\nu\vert \mathfrak{p}}$ sont majorées.

Pour cela, fixons $\nu$ une place de $L$ au-dessus de $\mathfrak{p}$. 
Pour $n\in\N$ et $\mathfrak{m}\in\Omega_n$, notons $\nu_{n,\mathfrak{m}}$ la place de $K_{n,\mathfrak{m}}$ au-dessous de $\nu$. 
 On va montrer que la famille des degrés locaux $[(K_{n,\mathfrak{m}})_{\nu_{n,\mathfrak{m}}} : K_\mathfrak{p}]$ est majorée (par un réel indépendant de $\nu, n$ et $\mathfrak{m}$). 
Car alors, les complétés $(K_{n,\mathfrak{m}})_{\nu_{n,\mathfrak{m}}}$ ($(n,\mathfrak{m})\in \N \times \Omega_n$) étant en nombre fini d'après le théorème \ref{Krasner}, on aura que $L_\nu$ est égal au compositum de tous les $(K_{n,\mathfrak{m}})_{\nu_{n,\mathfrak{m}}}$ et le corollaire \ref{corollaire immédiat du thm de Krasner} donnera bien l'existence d'une constante $C$ indépendante de $\nu,n,\mathfrak{m}$ telle que $e_\nu(L\vert K)\leq C$ et $f_\nu(L \vert K)\leq C$. 

Comme $[(K_{n,\mathfrak{m}})_{\nu_{n,\mathfrak{m}}} : K_\mathfrak{p}]=e_{\nu_{n,\mathfrak{m}}}(K_{n,\mathfrak{m}} \vert K)f_{\nu_{n,\mathfrak{m}}}(K_{n,\mathfrak{m}} \vert K)$, il suffit alors de majorer $e_{\nu_{n,\mathfrak{m}}}(K_{n,\mathfrak{m}} \vert K)$ et $f_{\nu_{n,\mathfrak{m}}}(K_{n,\mathfrak{m}} \vert K)$ par un réel indépendant de $\nu, n, \mathfrak{m}$. \\

 Commençons par $e_{\nu_{n,\mathfrak{m}}}(K_{n,\mathfrak{m}}\vert K_n)$. 
 Par hypothèse, $\mathfrak{m}$ et $p$ sont premiers entre eux. 
 Comme $\mathfrak{p}_n\mid p$, on en déduit que $\mathfrak{p}_n\nmid\mathfrak{m}$.
  Ainsi, d'après le $v)$ du théorème \ref{Takagi},  
\begin{equation} \label{propriété des corps de rayons sur la ramification.} 
 e_{\nu_{n,\mathfrak{m}}}(K_{n,\mathfrak{m}} \vert K_n)=1.
 \end{equation}

Intéressons nous maintenant au degré d'inertie $f_{\nu_{n,\mathfrak{m}}}(K_{n,\mathfrak{m}}\vert K_n)$.  
D'après le $v)$ du théorème \ref{Takagi}, $f_{\nu_{n,\mathfrak{m}}}(K_{n,\mathfrak{m}} \vert K_n)$ est le plus petit entier $h$ tel que $\mathfrak{p}_n^h\in P_\mathfrak{m}(K_n)$, i.e. tel que

\begin{enumerate} [i)]
\item  $\mathfrak{p}_n^h$ est principal. Notons $\gamma\in\mathcal{O}_{K_n}$ un générateur de $\mathfrak{p}_n^h$.
\item $\gamma\equiv 1 \; \mod \mathfrak{m}$,
\item pour tout plongement réel $\sigma : K_n\hookrightarrow\R$, on a $\sigma(\gamma)>0$.
\end{enumerate}
  Notons   
\begin{equation} \label{définition de N}  
  N= \ppcm(j \; \vert \; j=1,\dots,M \; , \pgcd(j,p)=1).
  \end{equation}
Notons également $f_\mathfrak{p}(K)$ l'ordre de $\mathfrak{p}$ dans $Cl(K)$ (le groupe des classes de $K$) et $(\alpha)$ l'idéal $\mathfrak{p}^{f_\mathfrak{p}(K)}$. 

Si $N\neq 1$, alors comme $\mathfrak{p}$ est au-dessus de $p$ et que $p\nmid N$, on a 
\begin{equation*}
\alpha\in \left(\mathcal{O}_K/N\mathcal{O}_K\right)^{\times}.
\end{equation*}
Ainsi, pour tout $N$, il existe un plus petit entier strictement positif $g$ tel que $\alpha^g \equiv 1 \mod N\mathcal{O}_K$ (on a $g=1$ si $N=1$). 

 Par hypothèse, il existe un unique idéal $\mathfrak{p}_n$ de $\mathcal{O}_{K_n}$ au-dessus de $\mathfrak{p}$.
Par conséquent, $\mathfrak{p}\mathcal{O}_{K_n}=\mathfrak{p}_n^{e_{\mathfrak{p}_n}(K_n\vert K)}$.
Ainsi, $\mathfrak{p}_n^{h_n}$ avec  
\begin{equation} \label{expression de h n}
h_n:=2 f_\mathfrak{p}(K)e_{\mathfrak{p}_n}(K_n\vert K)g 
\end{equation}
  vérifie les conditions $i)$, $ii)$ et $iii)$ ci-dessus.
  En effet, un rapide calcul montre que $\mathfrak{p}_n^{h_n}=(\alpha^{2g})$.
  Ainsi, $i)$ est vérifié. 
Comme $\alpha^g \equiv 1 \mod N\mathcal{O}_K$ et que $\mathfrak{m}\mid N\mathcal{O}_{K_n}$, il s'ensuit que $\alpha^{2g} \equiv 1 \mod \mathfrak{m}$, ce qui montre $ii)$.  
  Enfin, $iii)$ est vérifié car si $\sigma : K_n \hookrightarrow \R$, alors $\sigma(\alpha^{2g})=(\sigma(\alpha))^{2g}>0$.
  Ainsi, $\mathfrak{p}_n^{h_n}\in P_\mathfrak{m}(K_n)$.
  
 On a donc montré l'inégalité $f_{\nu_{n,\mathfrak{m}}}(K_{n,\mathfrak{m}} \vert K_n) \leq h_n$.  
 Par ailleurs, la suite $(d_n)_n$ est majorée par hypothèse. 
 Par conséquent, la suite $(e_{\mathfrak{p}_n}(K_n\vert K))_n$ est majorée. 
  De \eqref{expression de h n}, il devient alors clair que la suite $(h_n)_n$ est majorée. 
On a ainsi majoré $f_{\nu_{n,\mathfrak{m}}}(K_{n,\mathfrak{m}}\vert K_n)$ par une constante ne dépendant pas de $\nu$, $n$ et $\mathfrak{m}$.   
  Ceci prouve le théorème.
                   \qed

  Remarquons que si $M=1$, alors $K_{n,\mathfrak{m}}$ correspond au corps de Hilbert de $K_n$.
 Ainsi, le cas $M=1$ et $K=\Q$ de ce théorème correspond au théorème \ref{AmelioGal}.

 \begin{rqu} \label{remarque que j'ai dû rajouter}
En utilisant le théorème \ref{inertiecompositum} à la place du corollaire \ref{corollaire immédiat du thm de Krasner}, il est possible de donner une majoration explicite plus précise des familles $(e_\nu(L\vert \Q))_{\nu \mid \mathfrak{p}}$ et $(f_\nu(L\vert \Q))_{\nu \mid \mathfrak{p}}$. 
Puis, en utilisant le lemme \ref{lemme classique}, on en déduit une minoration explicite de la hauteur. 

On reprend les notations du théorème \ref{version quantitative du thm principal.}.
Pour $n\in\N$, posons $e_n=e_{\mathfrak{p}_n}(K_n\vert K)$.
Soit $\nu\in\mathcal{M}(L)$ au-dessus de $\mathfrak{p}$.
De \eqref{propriété des corps de rayons sur la ramification.}, on a $e_n=e_{\nu_{n,\mathfrak{m}}}(K_{n,\mathfrak{m}}\vert K)$ pour tout $\mathfrak{m}\in\Omega_n$. 
 
Notons $\Lambda$ l'ensemble des entiers $e_n$ divisibles par $p$. 
Si $\Lambda=\emptyset$, alors d'après le théorème \ref{inertiecompositum}, un majorant de $e_\nu(L\vert \Q)$ est 
\begin{equation} \label{cas de ramification non sauvage}
e:=e_\mathfrak{p}(K\vert\Q)\ppcm_{n\in\N}(e_n)
\end{equation}
(car nous n'avons pas d'extensions sauvagement ramifiées et donc $m=n$).
Remarquons que dans ce cas, $e$ ne dépend pas de $\nu$. 

Supposons maintenant que $\Lambda\neq \emptyset$. 
Soit $\tilde{e}\in \Lambda$. 
Alors d'après le théorème \ref{inertiecompositum}, un majorant de $e_\nu(L\vert K)$ est
\begin{equation} \label{calcul de e cas de ramification sauvage.}
e:=e_\mathfrak{p}(K\vert\Q)\underset{m\in\N,\; p\nmid e_m}{\ppcm}(e_m,\tilde{e})\;\tilde{e}^{\mathcal{N}(\tilde{e})-1}\prod\limits_{\underset{e'\neq \tilde{e}}{e'\in\Lambda}} {e'}^{\mathcal{N}(e')}
\end{equation} 
où $\mathcal{N}(e')$ désigne le nombre d'extensions $\{(K_{n,\mathfrak{m}})_{\nu_{n,\mathfrak{m}}}/K_\mathfrak{p} \vert \; n\in\N,\mathfrak{m}\in\Omega_n\}$ dont l'indice de ramification est précisément $e'$. 
Comme $e_{\nu_{n,\mathfrak{m}}}(K_{n,\mathfrak{m}}\vert K_n)=1$ (cf \eqref{propriété des corps de rayons sur la ramification.}), cela implique que $\mathcal{N}(e')$ est aussi le nombre d'extensions $\{K_n/K, n\in\N\}$  dont l'indice de ramification est précisément $e'$. 
Ainsi, $\mathcal{N}(e')$, et donc $e$, ne dépend pas de $\nu$.

Calculons maintenant un majorant de la famille $(f_\nu(L\vert K))_{\nu\mid \mathfrak{p}}$.
Notons, comme dans la preuve du théorème \ref{version quantitative du thm principal.}, $\alpha$ un générateur de $\mathfrak{p}^{f_\mathfrak{p}(K)}$.
Notons $g_\mathfrak{m}$ le plus petit entier strictement positif tel que $\alpha^{g_\mathfrak{m}}\equiv 1 \mod \mathfrak{m}$. 
	Définissons également $\epsilon_\mathfrak{m}$ comme suit : $\varepsilon_\mathfrak{m}$ vaut $1$ si $\sigma(\alpha^{g_\mathfrak{m}})>0$ pour tout plongement réel $\sigma : K_n\hookrightarrow\R$ et $2$ sinon. 
Par un calcul similaire à celui effectué dans la preuve du théorème \ref{version quantitative du thm principal.}, on en déduit que 
 \begin{equation*} 
 f_{\nu_{n,\mathfrak{m}}}(K_{n,\mathfrak{m}}\vert K_n)\mid \varepsilon_\mathfrak{m}\;f_\mathfrak{p}(K)\;e_ng_\mathfrak{m}.
 \end{equation*}
  Comme $d_n=e_nf_{\mathfrak{p}_n}(K_n\vert K)$, le théorème \ref{inertiecompositum} montre qu'un majorant de $f_\nu(L\vert \Q)$ est
\begin{equation} \label{expression de f}
f:=f_\mathfrak{p}(K\vert \Q) f_\mathfrak{p}(K)\underset{n\in\N, \mathfrak{m}\in\Omega_n}{\ppcm}(\varepsilon_\mathfrak{m}\;g_\mathfrak{m} d_n)E
\end{equation}
pour un certain $E$ explicitement calculable (cf \eqref{expression de E sans sauvage} et \eqref{expression de E avec sauvage}) ne dépendant pas de $\nu$. 
 
En reprenant les notations de la définition \ref{définitions de k }, on déduit du lemme \ref{lemme classique} que pour tout $x\in L^*\backslash \mu_\infty$,
\begin{equation*}
\h(x)\geq \frac{1}{p^{f+\lambda}+p^\lambda}\left(\frac{\beta[K_\mathfrak{p} : \Q_p]}{[K : \Q]}\log p-\log 2\right)>0
\end{equation*}
où $\lambda=\lambda(e,\mathfrak{p})$ et $\beta=\beta(e,\mathfrak{p})$. 
\end{rqu}

Remarquons que la minoration de la hauteur ci-dessus reste strictement positive si, à la place de $e$, on prend un majorant quelconque de la famille $(e_\nu(L\vert \Q))_{\nu\vert \mathfrak{p}}$.

\subsection{Exemples.}

Dans cette sous-section, on donne quatre exemples de corps ayant la propriété (B) ainsi qu'une minoration explicite de la hauteur. 
Les deux premiers sont des généralisations d'exemples déjà illustrés dans \cite{GalateauSmallheightinfieldsgeneratedbysingularmoduli}. 

Dans ces quatre exemples, on se place dans le cas particulier où $M=1$ et $K=\Q$. 
Par conséquent, $f_\mathfrak{p}(K)=1$.
Il s'ensuit donc que $(\alpha)=\mathfrak{p}=p\Z$. 
On peut ainsi supposer que $\alpha=p$.
Comme $M=1$, il s'ensuit que $\Omega_n=\{\mathcal{O}_{K_n}\}$ et donc que $\mathfrak{m}=(1)$. 
Ainsi, $K_{n,\mathfrak{m}}=H(K_n)$. 
De plus, par définition de $g_\mathfrak{m}$ et de $\epsilon_\mathfrak{m}$, il s'ensuit que $g_\mathfrak{m}=\epsilon_\mathfrak{m}=1$. 
De \eqref{expression de f}, il en résulte, dans cette situation, que $f=\ppcm_{n\in\N}(d_n)E$ où $d_n=[K_n : \Q]$ et où $E$ est défini comme dans \eqref{expression de E sans sauvage} ou \eqref{expression de E avec sauvage}.
 
 Pour un corps de nombres $K$, notons $\Delta_K$ son discriminant.

\begin{ex} \label{ex Galateau}
\rm{ Fixons un premier $p>3$.
Pour un entier $D$ non nul et sans facteur carré, notons $K_D=\Q(\sqrt{D})$. 
 Notons également
\begin{equation*}
\mathcal{D}=\{D\in\Z\backslash\{0\} \; \vert \; p\mathcal{O}_{K_D} \; \text{est premier ou totalement ramifié} \}
\end{equation*}
 et $L$ le compositum des $H(K_D)$ avec $D\in\mathcal{D}$.
 On souhaite minorer la hauteur des éléments de $L^*\backslash \mu_\infty$.
 
Le théorème \ref{Krasner} montre qu'il y a deux extensions quadratiques de $\Q_p$ totalement ramifiées (que l'on note $K_1$ et $K_2$) et une seule non ramifiée (que l'on note $K_3$). 
Il s'ensuit donc que $a_3(q)=0$ (cf \eqref{définition de a(q) r}) pour tout premier $q$.
Comme aucune extension n'est sauvagement ramifiée, on en déduit alors que $E=2$ (cf \eqref{expression de E sans sauvage}) et que $e=2$.
On a également $d_n=2$ pour tout $n$. 
Par conséquent, $f=4$.
Avec les notations de la définition \ref{définitions de k }, on obtient $\lambda(e,p)=0$ car $p\geq 5$ et $\beta(e,p)=1/2$.
La minoration de la hauteur que l'on a obtenue dans la remarque \ref{remarque que j'ai dû rajouter} montre donc que pour tout $x\in L^*\backslash \mu_\infty$, 
\begin{equation*}
\h(x)\geq \frac{\log(p/4)}{2(p^4+1)}.
\end{equation*}

Dans \cite[§ 5]{GalateauSmallheightinfieldsgeneratedbysingularmoduli}, l'auteur obtenait la minoration $(\log(p/2))/(p^2+1)$.
Cependant, il montra que cette minoration est valable non pas pour $L$, mais pour le compositum des $H(K_D)$ avec $D\in\mathcal{D}_0$ où 
\begin{equation*}
\mathcal{D}_0= \{D\in\N^* \; \vert \; D \; \text{est premier}, \; p\mathcal{O}_{K_{-D}} \; \text{est premier et}\; D\equiv 3 \mod 4\}.
\end{equation*}
Il doit supposer en plus que $D\equiv 3 \mod 4$ afin que $\Delta_{\Q(\sqrt{-D})}=-D$ et que $D=p$ est premier afin que les $\Delta_{\Q(\sqrt{-p})}$ soient deux à deux premiers entre eux.
}
\end{ex}

La différence entre nos deux minorations provient du fait que dans notre situation, $p$ se ramifie. 
Si $p$ ne se ramifie pas (comme c'est le cas si on considère l'ensemble des entiers $D$ non nuls et sans facteur carré tels que $p\mathcal{O}_{K_D}$ est premier), on a $e=1$, $E=1$ et $\beta(e,p)=1$. On obtient ainsi $(\log(p/2))/(p^2+1)$ comme minoration au lieu de $(\log(p/4))/(2(p^4+1))$.
Cette minoration a également été obtenue par Pottmeyer dans \cite[Theorem 2.3]{Smalltotally}. 

\begin{ex}
\rm{Dans \cite{SHANKSDANIELthesimplestcubicfields}, l'auteur étudie le corps de décomposition (\og simplest cubic field\fg) $K_m$ du polynôme 
\begin{equation*}
P_m(X)=X^3-mX^2-(m+3)X-1
\end{equation*}
avec $m\in\N$. 
Pour tout $m$, l'extension $K_m/\Q$ est galoisienne de degré $3$.
Ainsi, $K_m=\Q(x_m)$ où $x_m$ est une racine de $P_m$.
On a également que le discriminant $\rm{disc}(P_m)$ du polynôme $P_m$ vaut $(m^2+3m+9)^2$.

Montrons que $2$ est inerte dans tous les $K_m$.
Il est clair que $P_m$, vu comme un polynôme de $\F_2[X]$, est irréductible (c'est un polynôme de degré $3$ ne possédant pas de racine dans $\F_2$). 
Par conséquent, $P_m$ est irréductible dans $\Z_2[X]$.
Comme $K_m=\Q(x_m)$, on déduit d'un lemme de Kummer (cf \cite[Chapter 2, section 10]{FMJT2}) qu'il n'y a qu'un seul idéal premier au-dessus de $2$.
Comme $2$ ne se ramifie pas dans $K_m$ (car $\rm{disc}(P_m)$ est impair et que l'on a $\rm{disc(P_m)}=[\mathcal{O}_{K_m} : \Z[x_m]]^2\Delta_{K_m}$), il s'ensuit que $2$ est bien inerte dans $K_m$.

D'après le théorème \ref{version quantitative du thm principal.}, le compositum $L$ de tous les $H(K_m)$ a la propriété (B). 
Nous pouvons être plus explicite. 
Ici, on est dans le cas où $p=2$. 
Comme il n'y a pas de ramification, on a $e=1$ d'après \eqref{cas de ramification non sauvage} et $E=1$ d'après \eqref{expression de E sans sauvage}. 
On a également $d_n=3$ pour tout $n$. 
Par conséquent, $f=3$.
Enfin, avec les notations de la définition \ref{définitions de k }, on a $\lambda(e,p)=1$ et donc $\beta(e,p)=2$.
La minoration de la hauteur que l'on a obtenue montre donc que pour tout $x\in L^*\backslash \mu_\infty$, 
\begin{equation*}
\h(x)\geq \frac{\log 2}{18}.
\end{equation*}

\`A cause des restrictions nécessaires pour appliquer son théorème, Galateau s'intéresse aux corps $K_m$ dont l'anneau des entiers est $\Z[x_m]$, ce qui permet d'en déduire la relation $\rm{disc}(P_m)=\Delta_{K_m}$ et donc que $\Delta_{K_m}=(m^2+3m+9)^2$.
Il a ensuite montré qu'il existe un ensemble infini $\mathcal{N}$ tel que l'anneau des entiers de $K_m$, pour $m\in\mathcal{N}$, est $\Z[x_m]$ et tel que les $(\Delta_{K_m})_{m\in\mathcal{N}}$ sont deux à deux premiers entre eux.
 Il peut ainsi appliquer son théorème et obtenir que le compositum $L_0$ des $H(K_m)$ avec $m\in\mathcal{N}$ a la propriété (B).
Pour ce corps, il obtient la même minoration que la nôtre. }
\end{ex}

Notre théorème permet d'obtenir d'autres exemples.
Nous allons en voir deux.
Le premier ne contient que des extensions non sauvagement ramifiées tandis que le second en contiendra.

Le lecteur pourra remarquer que la minoration de la hauteur dans l'exemple \ref{exemple sans ramification sauvage} est de bien meilleure qualité que dans l'exemple \ref{exemple avec ramification sauvage}. 

\begin{ex} \label{exemple sans ramification sauvage}
\rm{Plaçons nous dans le cas où $p=3$.
Soit $S$ l'ensemble des corps de degrés $2$, $4$ ou $5$ tel que pour tout $K\in S$, il existe un unique idéal premier $\mathfrak{p}_K$ de $\mathcal{O}_K$ au-dessus de $3$. 
Notons $L$ le compositum de tous les $H(K)$ pour $K\in S$.
 
Comme la suite $([K_{\mathfrak{p}_K} : \Q_3])_{K\in S}$ est majorée par $5$, le théorème \ref{Krasner} montre que l'ensemble $\{K_{\mathfrak{p}_K}, K\in S\}$ est fini.
Notons $K_1,\dots,K_n$ les localisés $K_{\mathfrak{p}_K}$ avec $K\in S$.

 Soit $\nu$ une place de $\mathcal{O}_L$ au-dessus de $3$ ($\nu$ est donc au-dessus de chacun des $\mathfrak{p}_K$ par unicité de l'idéal premier au-dessus de $3$).
Comme $K_j/\Q_3$ est non sauvagement ramifié pour tout $j\leq n$, il découle de \eqref{cas de ramification non sauvage} que
\begin{equation*}
e=\ppcm(1,2,4,5)=20.
\end{equation*}
Comme $2\;3^{3-1}\leq 20 < 2\; 3^3$, on en déduit, avec les notations de la définition \ref{définitions de k }, que $\lambda(e,p)=3$ et donc que $\beta(e,p)=1,35$. 

 Enfin, $f(K_j\vert\Q_3)e(K_j\vert\Q_3)\leq 5$.
On obtient donc que 
\begin{equation*}
\begin{cases}
f(K_j \vert \Q_3)=1 \; \text{si} \; e(K_j\vert\Q_3)\geq 3 \\
f(K_j \vert \Q_3)\in\{1,2\} \; \text{si} \; e(K_j\vert\Q_3)= 2 \\
f(K_j \vert \Q_3)\leq 5 \; \text{si} \; e(K_j\vert\Q_3)=1 
\end{cases} .
\end{equation*}
 On déduit alors de \eqref{majoration du nombre N(e).} et du théorème \ref{Krasner} que $\mathcal{N}(2)\leq 2+2=4;\;\mathcal{N}(4)\leq 4$ et $\mathcal{N}(5)\leq 5$.
 Enfin, $a_n(2)=1$ et $a_n(q)=0$ pour les premiers impairs (cf \eqref{définition de a(q) r}).
  On en déduit que  
\begin{equation*}
f\leq \ppcm(2,4,5) \; 2^3 \; 4^3 \; 5^4 \; 2=20\;2^3\;4^3\;5^4\; 2= 1.28 \;10^7.
\end{equation*}

Ainsi, pour tout $x\in L^*\backslash\mu_\infty$, on a
\begin{equation*}
\h(x)\geq\frac{0.78}{3^{f+3}+3^3}\geq e^{-1.41 \; 10^7}. 
\end{equation*}
}
\end{ex}

\begin{ex} \label{exemple avec ramification sauvage}
\rm{Plaçons nous de nouveau dans le cas où $p=3$.
Soit $S$ l'ensemble des corps de degré $\leq 5$ tel que pour tout $K\in S$, il existe un unique idéal premier $\mathfrak{p}_K$ de $\mathcal{O}_K$ au-dessus de $3$. 
Notons $L$ le compositum de tous les $H(K)$ pour $K\in S$.
 
Comme la suite $([K_{\mathfrak{p}_K} : \Q_3])_{K\in S}$ est majorée par $5$, le théorème \ref{Krasner} montre que l'ensemble $\{K_{\mathfrak{p}_K}, K\in S\}$ est fini.
Notons $K_1,\dots,K_n$ les localisés $K_{\mathfrak{p}_K}$ avec $K\in S$.

 Notons $m$ le nombre d'extensions $K_j/\Q_3$ non sauvagement ramifiées.
Quitte à renuméroter, on peut supposer que ces $m$ extensions sont $K_1/\Q_3,\dots,K_m/\Q_3$.
Les extensions sauvagement ramifiées ont un indice de ramification égal à $3$. 
Ainsi, (cf \eqref{définition de Lambda r} pour la définition de $\Lambda_r$)
\begin{equation*}
\Lambda_{m+1}=\Lambda_{m+2}=\Lambda_n=\{1,\dots,5\}.
\end{equation*}
 Comme $f(L_j\vert\Q_3)e(L_j\vert\Q_3)\leq 5$, on en déduit que
\begin{equation*}
\begin{cases}
f(K_j \vert \Q_3)=1 \; \text{si} \; e(K_j\vert\Q_3)\geq 3 \\
f(K_j \vert \Q_3)\in\{1,2\} \; \text{si} \; e(K_j\vert\Q_3)= 2 \\
f(K_j \vert \Q_3)\leq 5 \; \text{si} \; e(K_j\vert\Q_3)=1 
\end{cases} .
\end{equation*}
Par conséquent, le \eqref{majoration du nombre N(e).} et le théorème \ref{Krasner} montrent que $\mathcal{N}(2)\leq 2+2=4,\;\mathcal{N}(3)\leq 21,\;\mathcal{N}(4)\leq 4$ et $\mathcal{N}(5)\leq 5$.
 Enfin, $a_{m+2}(2)=1$ et $a_{m+2}(q)=0$ pour les premiers impairs $q$.
  Par conséquent,   
\begin{equation*}
f\leq \ppcm(1,2,\dots,5) \frac{1}{5!}\; 2\;2^4\;3^{21}\;4^4\;5^5.
\end{equation*}
De plus, d'après \eqref{calcul de e cas de ramification sauvage.}, $\ppcm(1,2,\dots,5) \; 3^{\mathcal{N}(3)-1}$ est un majorant de la famille $(e_\nu(L \vert \Q))_{\nu \mid 3}$.
Il s'ensuit donc que $e_1=60 \;3^{20}$ est un majorant de la famille $(e_\nu(L \vert \Q))_{\nu \mid 3}$.

Avec les notations de la définition \ref{définitions de k }, on a $\lambda(e_1,p)=24$.
Ainsi, $\beta(e_1,p)=1,35$ et donc, pour tout $x\in L^*\backslash\mu_\infty$, on a
\begin{equation*}
\h(x)\geq\frac{0.78}{3^{f+24}+3^{24}}\geq e^{-3.6 \; 10^{18}}. 
\end{equation*}
}
\end{ex}

\section{Un problème de densité.}

Fixons un premier $q$.
Pour un corps $K$, on notera $\Delta_K$ son discriminant.
Dans \cite{GalateauSmallheightinfieldsgeneratedbysingularmoduli}, Galateau s'est intéressé à la densité de corps quadratiques tels que $q$ est inerte dans chacun de ces corps et tels que leurs discriminants sont deux à deux premiers entre eux, conditions nécessaires dans sa preuve.  
Comme un corps quadratique est de la forme $\Q(\sqrt{D})$ pour un certain entier $D$ non nul sans facteur carré, cela revient à déterminer la densité naturelle d'ensembles d'entiers sans facteur carré $D$ vérifiant les conditions de l'hypothèse ci-dessous, pour $q$ un premier fixé :

\begin{hyp} \label{hyp}
\begin{enumerate}[i)] ~
\item $q$ est inerte dans $\Q(\sqrt{D})$;
\item  les $\Delta_{\Q(\sqrt{D})}$ sont deux à deux premiers entre eux. 
\end{enumerate}
\end{hyp}

Cette densité valant $0$ puisque les $D$ doivent être deux à deux premiers entre eux, l'auteur considère plutôt la densité de Dirichlet des premiers $D$ vérifiant les conditions de l'hypothèse \ref{hyp}. 

Notons $\mathcal{P}$ l'ensemble des nombres premiers.
Soit $X$ un sous-ensemble de $\mathcal{P}$. Si la limite 
\begin{equation*}
\lim_{s\rightarrow 1^+}\frac{-1}{\log(s-1)}\sum_{p\in X} \frac{1}{p^s}
\end{equation*}
existe, alors on dira que $X$ admet une densité de Dirichlet, notée $\mathcal{D}(X)$, égale à la valeur de cette limite.
Par exemple, si $X$ est fini, alors $\mathcal{D}(X)=0$ et $\mathcal{D}(\mathcal{P})=1$ (théorème de Dirichlet).

Galateau a montré:

\begin{prop} (\cite[Lemma 3.1]{GalateauSmallheightinfieldsgeneratedbysingularmoduli})
Il existe un ensemble de premiers $\mathcal{P}_q$, de densité de Dirichlet égale à $\frac{1}{4}$, tel que les entiers $D:=-p$ avec $p\in\mathcal{P}_q$ vérifient les conditions de l'hypothèse \ref{hyp}.
\end{prop}

Dans notre situation, nous n'avons plus la condition $ii)$ de l'hypothèse \ref{hyp} qui était nécessaire dans \cite{GalateauSmallheightinfieldsgeneratedbysingularmoduli}. 
Cela nous permet de considérer la densité naturelle dont on rappelle ci-dessous la définition en toute généralité.

Le but de cette section est de montrer que dans certains cas, cette densité peut être strictement positive. 
Même si cela semble naturel, cela repose sur des résultats profonds. 

Pour $d\in\N^*$ et $n\in\N^*$, notons $N_n(d)$ le nombre de corps de degré $n$ sur $\Q$ et de discriminants bornés, en valeur absolue, par $d$.

\begin{defn}
Soit $E_n$ un ensemble de corps de degré $n$ fixé. On note $E_n(d)$ l'ensemble des corps de $E_n$ de discriminants bornés, en valeur absolue, par $d$. On dira que $E_n$ a une densité naturelle si $\frac{\#E_n(d)}{N_n(d)}$ a une limite quand $d$ tend vers $+\infty$ et on la notera, si elle existe, $d(E_n)$.

\end{defn}
Notons $\mathcal{I}(p,n)$ (resp. $\mathcal{R}(p,n)$) l'ensemble de tous les corps de degré $n$  dans lesquels $p$ est inerte (resp. totalement ramifié).

Supposons maintenant $p>2$.
L'existence (et donc le calcul) de $d(\mathcal{R}(p,n))$ et de $d(\mathcal{I}(p,n))$ est un problème non-trivial sauf pour le cas $n=2$ qui a été traité par Gauss. 
Pour le cas $n=3$, la preuve repose sur une propriété propre aux corps cubiques (cf \cite{DavenportHeilbronnOnthedensityofdiscriminantsofcubicfields}).
Les cas $n=4$ et $n=5$ ont été traités par Bhargava et reposent sur des résultats très profonds (cf \cite{BhargavaThedensityofdiscriminantsofquarticringsandfields} et \cite{BhargavaThedensityofdiscriminantsofquinticringsandfields}).
Pour la commodité du lecteur, nous rappelons ces résultats.

   \begin{thm} [Gauss]
  Pour $n=2$, on a   
\begin{equation*}  
  d(\mathcal{R}(p,2))=\frac{1}{p+1}
  \end{equation*}
   et   
\begin{equation*}   
   d(\mathcal{I}(p,2))=\frac{p}{2(p+1)}.
   \end{equation*}
   \end{thm}
  
\begin{thm}[Davenport-Heilbronn]

Pour $n=3$, on a 
\begin{equation*}
d(\mathcal{R}(p,3))=\frac{1}{p^2+1}
\end{equation*} 
et 
\begin{equation*}
d(\mathcal{I}(p,3))=\frac{p(p-1)}{3(p^2+1)}.
\end{equation*}
\end{thm}

\begin{thm}[Bhargava]

Pour $n=4$, 
\begin{equation*}
d(\mathcal{I}(p,4))=\frac{1}{4}\left(1-\frac{(p+1)^2}{p^3+p^2+2p+1}\right).
\end{equation*}
 Pour $n=5$, 
\begin{equation*} 
 d(\mathcal{I}(p,5))=\frac{1}{5}\left(1-\frac{(p+1)(p^2+p+1)}{p^4+p^3+2p^2+2p+1}\right).
 \end{equation*}
\end{thm}

Ces différents théorèmes montrent que la densité de corps $K_m$ de degré $n\leq 5$ fixé sur $\Q$ tel qu'il existe un unique idéal premier $\mathfrak{p}_m$ de $\mathcal{O}_{K_m}$ au-dessus de $p$ est strictement positive.
Cela montre que notre théorème principal permet de considérer beaucoup plus de corps que le théorème \ref{Gal(B)}.

De ces résultats, nous pouvons conjecturer que:

\begin{conj} \label{conj nono}
Pour tout $n$, pour tout $p$, $\mathcal{I}(p,n)$ a une densité naturelle et 
\begin{equation*}
\lim\limits_{p\rightarrow +\infty} d(\mathcal{I}(p,n))=\frac{1}{n}.
\end{equation*}
\end{conj}

Il n'existe pas de résultats connus pour $n\geq 6$.
 Del Corso et Dvornicich ont étudié dans \cite{DelCorsoIlariaandDvornicichRoberto} un autre type de densité.

Notons $\Phi_n(N)\subset \Z[X]$ l'ensemble des polynômes irréductibles de degré $n$ fixé dont les coefficients sont majorés en valeur absolue par $N>0$. 
Soit $\Phi_n\subset\Z[X]$ un ensemble de polynômes irréductibles de degré $n$.
Si la limite
\begin{equation*}
\lim_{N\rightarrow +\infty}\frac{\mid \Phi_n\cap\Phi_n(N)\mid}{\mid \Phi_n(N)\mid }
\end{equation*}
existe, on dit que $\Phi_n$ admet une densité naturelle, notée $D(\Phi_n)$, égale à la valeur de cette limite.

Soient $e,f\in\N^*$.  
Notons $\mathcal{A}(p;(e,1))$ (resp. $\mathcal{A}(p; (1,f))$), l'ensemble des polynômes unitaires et irréductibles $g$ de degré $e$ (resp. $f$) tels que $p$ est totalement ramifié  (resp. inerte) dans $\Q[X]/(g(X))$.
On a alors 

\begin{thm}[\cite{DelCorsoIlariaandDvornicichRoberto}, Main Theorem]
Il existe deux fonctions rationnelles $\phi_e(X)$, $\; \phi_f(X)\in\Q(X)$ telles que $D(\mathcal{A}(p,(e,1))$ existe et vaut $\phi_e(p)$ pour tout premier $p$ premier à $e$ et telles que $D(\mathcal{A}(p,(1,f))$ existe et vaut $\phi_f(p)$ pour tout premier $p$.
\end{thm}

Le lecteur intéressé pourra consulter \cite{DelCorsoIlariaandDvornicichRoberto} pour des résultats plus généraux. 

\section{Appendice.} \label{appendice}

Dans cet appendice, nous allons montrer le théorème \ref{inertiecompositum}.
Pour une extension $L/K$ de corps locaux, les deux diagrammes  
 \begin{equation*}
 \xymatrix@C=4pc@R=4pc{
L\ar@{-}[d]^(0.3){e}^(0.6){f} \\
 K }\;\text{et} \; \xymatrix@C=4pc@R=4pc{
L\ar@{-}[d]^(0.5){d} \\
 K }
\end{equation*} 
  signifient que $e=e(L\vert K)$, $f=f(L\vert K)$ et $d=[L : K]$.
  
  Fixons un nombre premier $p$ et une extension finie $F/\Q_p$. 
  Soient $n\geq 2$ et $K_1/F,\dots,K_n/F$ des extensions deux à deux distinctes.
Posons $e_i=e(K_i\vert F)$ et $f_i=f(K_i\vert F)$.

Commençons par majorer le degré d'inertie d'un compositum de deux corps.

\begin{prop} \label{inertie2corps}
On a
\begin{equation*}
f(K_1K_2\vert F)\leq \ppcm(f_1,f_2)\;\pgcd(e_1, e_2).
\end{equation*}
\end{prop}   
 
 \begin{proof}
 Notons $K_i^{nr}$ l'extension maximale non-ramifiée de $F$ incluse dans $K_i$.
 Par \cite[chapitre 3, théorème 25]{FMJT}, $K_i/K_i^{nr}$ est totalement ramifiée de degré $e(K_i\vert K_i^{nr})= e_i$ et donc $f(K_i\vert K_i^{nr})=1$. 
 Par conséquent, $f(K_i^{nr}\vert F)=f_i$.

 Considérons le diagramme suivant: 
 
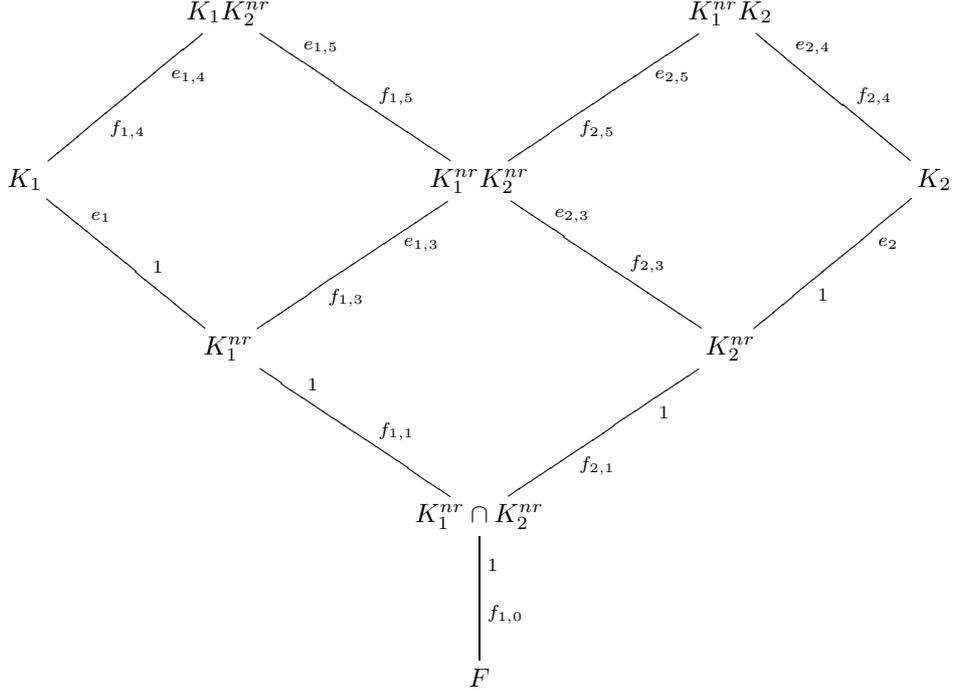
\begin{equation*}
\xymatrix@C=4pc@R=4pc{
 & K_1K_2^{nr}\ar@{-}[dr]^(0.3){e_{1,5}}^(0.6){f_{1,5}}\ar@{-}[dl]^(.3){e_{1,4}}^(0.6){f_{1,4}} & & K_1^{nr}K_2\ar@{-}[dr]^(0.3){e_{2,4}}^(0.6){f_{2,4}}\ar@{-}[dl]^(.3){e_{2,5}}^(0.6){f_{2,5}} & \\
 K_1\ar@{-}[dr]^(0.3){e_1}^(0.6){1} & & K_1^{nr}K_2^{nr}\ar@{-}[dr]^(0.3){e_{2,3}}^(0.6){f_{2,3}}\ar@{-}[dl]^(.3){e_{1,3}}^(0.6){f_{1,3}} && K_2\ar@{-}[dl]^(0.3){e_2}^(0.6){1} \\
 & K_1^{nr}\ar@{-}[dr]^(0.3){1}^(0.6){f_{1,1}}\ar@{-} & & K_2^{nr}\ar@{-}[dl]^(.3){1}^(0.6){f_{2,1}} & \\
 & & K_1^{nr}\cap K_2^{nr}\ar@{-}[d]^(0.3){1}^(0.6){f_{1,0}} & \\
 & & F & }
\end{equation*}

 Comme le compositum de deux extensions non ramifiées est non ramifié, on a $e_{1,3}=e_{2,3}=1$.
  Comme l'extension $K_1^{nr}K_2^{nr}/K_i^{nr}$ est non ramifiée, on déduit du théorème \ref{Abhyankar} que $K_1K_2^{nr}/K_1$ est non ramifié et donc $e_{1,4}=1$. 
  Par conséquent, 
   \begin{equation*}
   e_{1,5}=e_{1,5}\;e_{1,3}=e_{1,4}\;e_1=e_1.
\end{equation*}   
  De même, $ e_{2,4}=1$ et $e_{2,5}=e_2$.
   
   Calculons maintenant les $f_i$. Cela se fait en utilisant le lemme suivant:
   
   \begin{claim*} 
   On a $\pgcd(f_{1,1},f_{2,1})=1$.
   \end{claim*}
   
   \begin{proof}
  Les trois extensions
   \begin{equation*}
   K_1^{nr}/K_1^{nr}\cap K_2^{nr},~ K_2^{nr}/K_1^{nr}\cap K_2^{nr}~\text{et} ~ K_1^{nr}K_2^{nr}/K_1^{nr}\cap K_2^{nr}
   \end{equation*}
    sont non-ramifiées. 
    Ce sont donc des extensions cycliques. Or,  
   \begin{equation} \label{isomorphsime}
\Gal(K_1^{nr}K_2^{nr}/K_1^{nr}\cap K_2^{nr})\simeq \Gal(K_1^{nr}/K_1^{nr}\cap K_2^{nr})\times \Gal(K_2^{nr}/K_1^{nr}\cap K_2^{nr}).    
   \end{equation}

On a donc un groupe cyclique qui est isomorphe à un produit cartésien de deux groupes cycliques. Il en résulte donc que le cardinal de $\Gal(K_1^{nr}/K_1^{nr}\cap K_2^{nr})$ et de $\Gal(K_2^{nr}/K_1^{nr}\cap K_2^{nr})$ sont premiers entre eux. Ceci prouve le fait.

   \end{proof}
   
Comme $f_1=f_{1,1}f_{1,0}$ et $f_2=f_{2,1}f_{1,0}$, on déduit du fait que $f_{1,0}=\pgcd(f_1, f_2)$ et $\ppcm(f_{1,1}, f_{2,1})=f_{1,1}f_{2,1}$. 
En passant aux cardinaux dans \eqref{isomorphsime}, on en déduit que $f_{1,3}f_{1,1}=f_{1,1}f_{2,1}$ et donc que $f_{1,3}=f_{2,1}$.

De l'inégalité $[K_1K_2^{nr} : K_1^{nr}K_2^{nr}] \leq [K_1 : K_1^{nr}]$, il en résulte que 
\begin{equation*}   
  e_1\; f_{1,5}= e_{1,5}\;f_{1,5}\leq e_1
\end{equation*}  
   et donc $f_{1,5}=1$. Enfin, 
\begin{equation*}   
    f_{1,4}=f_{1,3}\;f_{1,5}=f_{2,1}. 
\end{equation*}    
    Par symétrie, on a également $f_{2,4}=f_{2,3}=f_{1,1}$ et $f_{2,5}=1$.
   
\'A l'aide des valeurs calculées, on peut maintenant majorer le degré d'inertie $f(K_1K_2 \vert K_1^{nr}K_2^{nr})$. 
Considérons le diagramme ci-dessous :   
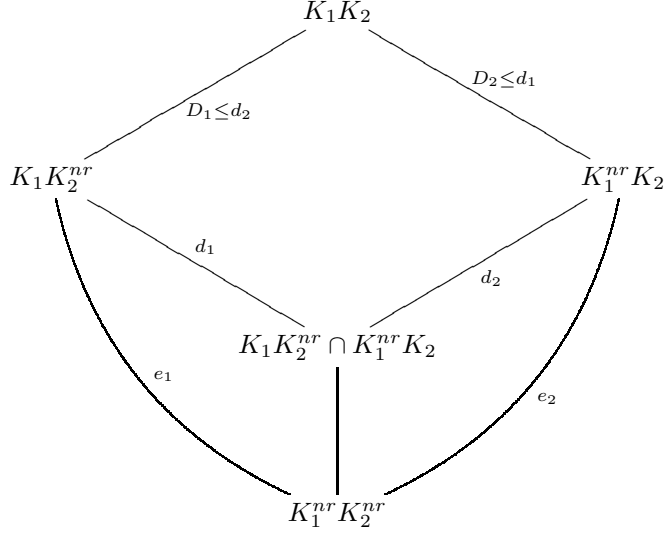
\begin{equation*}
\xymatrix@C=4pc@R=4pc{
 & & K_1K_2\ar@{-}[dr]^(0.5){D_2\leq d_1}\ar@{-}[dl]^(.5){D_1\leq d_2} && \\
 & K_1K_2^{nr}\ar@{-}[dr]^(0.5){d_1}\ar@{-}@/_2 pc/[drd]^(0.5){e_1} & & K_1^{nr}K_2\ar@{-}[dl]^(.5){d_2}\ar@{-}@/^2 pc/[dld]^(0.5){e_2} & \\
 & & K_1K_2^{nr}\cap K_1^{nr}K_2\ar@{-}[d] & \\
 & & K_1^{nr}K_2^{nr} & }
\end{equation*}
On remarque que $[K_1K_2 : K_1^{nr}K_2^{nr}]\leq d_2\; e_1$.
Il s'ensuit donc que:
\begin{equation*}
f(K_1K_2\vert K_1^{nr}K_2^{nr})\leq  \frac{d_2\; e_1}{e(K_1K_2\vert K_1^{nr}K_2^{nr})} .
\end{equation*}

Or, $\ppcm(e_1, e_2)\leq e(K_1K_2\vert K_1^{nr}K_2^{nr})$ et $d_2\leq e_2$. 
Ainsi,
\begin{equation*}
f(K_1K_2\vert K_1^{nr}K_2^{nr})\leq   \frac{e_1\; e_2}{\ppcm(e_1, e_2)}=\pgcd(e_1, e_2).
\end{equation*}

 Donc, $f(K_1K_2\vert F)\leq f_{1,1}f_{2,1}\pgcd(e_1,e_2)$ et la proposition s'ensuit puisque
\begin{align*}
 \pgcd (f_1,f_2)\;f_{1,1}\;f_{2,1} &= f_{1,0}f_{1,1}f_{2,1} \\
 & = \frac{f_1f_2}{\pgcd(f_1,f_2)}=\ppcm(f_1,f_2).
 \end{align*}      
    Ceci termine la preuve de la proposition \ref{inertie2corps}.  
\end{proof}      
Pour prouver le théorème \ref{inertiecompositum}, nous avons également besoin du lemme arithmétique élémentaire suivant.

\begin{lmm} \label{cor avec trop de calculs}
Soient $l\geq 2$ un entier et $a_1,...,a_l\in\N^*$. Alors 
\begin{equation} \label{l idée de Francesco en passant à la valuation}
\prod_{i=1}^{l-1} \pgcd(\ppcm(a_1,...,a_i),a_{i+1})=\underset{j=1,...,l}{\pgcd}\left(\prod\limits_{\underset{i\neq j}{i=1}}^{l}a_i\right).
\end{equation}
\end{lmm}

\begin{proof}
Soit $q$ un nombre premier. 
Afin d'alléger les notations, notons $P(i)=\max\{v_q(a_1),\dots,v_q(a_i)\}$ pour tout entier $i$.
En passant à la valuation $v_q$ dans \eqref{l idée de Francesco en passant à la valuation}, on en déduit que pour montrer le lemme, il suffit de montrer que 
\begin{equation} \label{mon hypothèse de récurrence}
\sum_{i=1}^{l-1} \min\{P(i), v_q(a_{i+1})\}=\sum_{i=1}^l v_q(a_i)-P(l)
\end{equation}
Il est clair que 
\begin{equation} \label{encore une simplification}
\min\{P(i), v_q(a_{i+1})\}+\max\{P(i), v_q(a_{i+1})\}=P(i)+v_q(a_{i+1}).
\end{equation}
Or, $\max\{P(i), v_q(a_{i+1})\}=P(i+1)$. 
Ainsi, en sommant chaque terme de l'égalité \eqref{encore une simplification} pour $i$ allant de $1$ à $l-1$, il s'ensuit que \[\sum_{i=1}^{l-1} \min\{P(i), v_q(a_{i+1})\}+\sum_{i=1}^{l-1} (P(i+1)-P(i))=\sum_{i=1}^{l-1} v_q(a_{i+1}).\] 
Comme $P(1)=v_q(a_1)$, un rapide calcul permet d'en déduire \eqref{mon hypothèse de récurrence}. Ceci termine la preuve du lemme. 
 
\end{proof}

Quitte à permuter les corps $K_1,\dots,K_n$, on peut supposer que $K_1/F,..., K_m/F$ sont non sauvagement ramifiés et que $K_{m+1}/F,...,K_n/F$ sont sauvagement ramifiés.

Montrons maintenant la proposition suivante : 

\begin{prop} \label{prop inutile}
On a
\begin{equation*}
f(K_1...K_n\vert F)\leq\ppcm(f_1,f_2,...,f_n) E
\end{equation*}
où
\begin{equation*}
E=\underset{j=1,\dots,n}{\pgcd}\left(\prod_{\underset{i\neq j}{i=1}}^n e_i\right)
\end{equation*}

si $m\geq n-2$ et
\begin{equation*}
E = \underset{j=1,\dots,m+2}{\pgcd}\left(\prod_{\underset{i\neq j}{i=1}}^{m+2} e_i\right)\prod_{i=m+3}^n e_i
\end{equation*}
si $m< n-2$. 

\end{prop}

\begin{proof}
 Pour tout $i$, notons $F_i:= f(K_1K_2...K_i\vert F)$ et $E_i:= e(K_1K_2...K_i\vert F)$. 
 D'après la proposition \ref{inertie2corps} où l'on prend le compositum $K_1\dots K_{i-1}$ pour corps $K_1$ et $K_i$ pour corps $K_2$, on a, pour tout $i\geq 2$,  
\begin{equation*} 
 F_i=\ppcm(F_{i-1},f_i) F'_i
 \end{equation*}
  avec  
  \begin{equation*}
  F'_i\in\{1,...,\pgcd(E_{i-1}, e_i)\}.
\end{equation*}
  Ainsi,   
  \begin{align*}
F_n & \leq \ppcm(F_{n-1},f_n)\pgcd(E_{n-1},e_n) \\
 & = \ppcm(\ppcm(F_{n-2},f_{n-1})\; F'_{n-1},f_n)\pgcd(E_{n-1},e_n).
\end{align*}
  De plus, pour tous entiers $n$ et $a,a_1,...,a_n\in\N^*$,  
\begin{equation*}  
  \ppcm(a\;a_1,a_2,...,a_n)\leq a\;\ppcm(a_1,...,a_n).
  \end{equation*}
On en déduit alors que
\begin{align*}
F_n & \leq F'_{n-1}\;\ppcm(F_{n-2}, f_{n-1}, f_n)\pgcd(E_{n-1},e_n) \\
 & \leq \ppcm(F_{n-2}, f_{n-1}, f_n) \pgcd(E_{n-2}, e_{n-1})\pgcd(E_{n-1},e_n)
\end{align*}
car $F'_{n-1}\leq\pgcd(E_{n-2}, e_{n-1})$.
Par récurrence descendante, 
\begin{equation*}
F_n\leq \ppcm(f_1,...,f_n)\prod\limits_{i=2}^n \pgcd(E_{i-1}, e_i).
\end{equation*}

Supposons $m\geq n-2$. 
Par le théorème \ref{Abhyankar}, $E_k=\ppcm(e_1,...,e_k)$ pour tout $k\leq n-1$. 
Il s'ensuit que 
\begin{equation*}
F_n\leq \ppcm(f_1,...,f_n)\prod\limits_{i=2}^n \pgcd(\ppcm(e_1,...,e_{i-1}), e_i)
\end{equation*}
et la première assertion du théorème se déduit du lemme \ref{cor avec trop de calculs}.

Supposons maintenant que $m< n-2$. 
Remarquons que \[\prod_{i=2}^n \pgcd(E_{i-1}, e_i)=\prod_{i=2}^{m+2} \pgcd(E_{i-1}, e_i)\prod_{i=m+3}^n \pgcd(E_{i-1}, e_i).\]
D'après le théorème \ref{Abhyankar}, on a $E_k=\ppcm(e_1,\dots,e_k)$ pour tout $k\leq m+1$.
On déduit ainsi du lemme \ref{cor avec trop de calculs} que
\begin{equation*}
F_n\leq \ppcm(f_1,...,f_n)\underset{j=1,...,m+2}{\pgcd}\left(\prod\limits_{\underset{i\neq j}{i=1}}^{m+2}e_i\right)  \prod\limits_{i=m+3}^{n}\pgcd(E_{i-1}, e_i)
\end{equation*}
et la seconde assertion du théorème s'ensuit car $\pgcd(E_{i-1}, e_i)\leq e_i$. 
\end{proof}

Pour la commodité du lecteur, rappelons les différentes notations utilisées pour le théorème \ref{inertiecompositum}.

Pour tout $r\in\{1,\dots,n\}$, notons
\begin{equation*}
\Lambda_r =\{e_1,\dots,e_r\}. 
\end{equation*}
Pour $e\in\Lambda_n$, notons $\mathcal{N}(e)$ le nombre d'extensions $K_i/F$ d'indice de ramification  $e$.
Enfin, pour tout premier $q$, notons
\begin{equation*}  
a_r(q):=\left(\sum\limits_{e\in\Lambda_r} v_q(e)\right)-\underset{e\in\Lambda_r}{\max}\{v_q(e)\}.
\end{equation*}

\begin{thm*} 
Si $n=m$, alors $e(K_1\dots K_n\vert F)=\ppcm(e_1,\dots,e_n)$. 
Si $m<n$, alors 
\begin{equation*}
e(K_1\dots K_n\vert F)\leq \ppcm(e_1,\dots,e_{m+1})  e_{m+1}^{\mathcal{N}(e_{m+1})-1}\prod_{e\in\Lambda_n\backslash\Lambda_{m+1}} e^{\mathcal{N}(e)}.
\end{equation*}
On a également
\begin{equation*}
f(K_1...K_n\vert F)\leq\ppcm(f_1,f_2,...,f_n) E
\end{equation*}
où
\begin{equation*}
E=\prod\limits_{e\in\Lambda_n} e^{\mathcal{N}(e)-1}\prod\limits_{q\in\mathcal{P}} q^{a_n(q)}
\end{equation*}
si $m\geq n-2$ et
\begin{equation*}
E = \prod\limits_{e\in\Lambda_{m+2}}e^{-1}\prod\limits_{q\in\mathcal{P}} q^{a_{m+2}(q)}\prod\limits_{e\in\Lambda_n} e^{\mathcal{N}(e)} 
\end{equation*}
si $m< n-2$. 
\end{thm*}

\begin{proof}

Rappelons que la valeur de $e(K_1\dots K_n\vert F)$ dans le cas $n=m$ a déjà été montrée juste après l'énoncé du théorème \ref{inertiecompositum}.

Montrons la majoration de l'indice de ramification dans le cas où $m<n$.
Par le théorème \ref{Abhyankar}, l'indice de ramification du compositum de toutes les extensions $K_i/F$ non-sauvagement ramifiées et de $K_{m+1}$ est  
 \begin{equation*}
\ppcm(e_1,\dots,e_{m+1}).
 \end{equation*}
 Par ailleurs, l'indice de ramification d'un compositum de corps est majoré par le produit des indices de ramification.
Ainsi, l'indice de ramification du compositum de toutes les extensions sauvagement ramifiées sauf $K_{m+1}$ est majoré par
\begin{equation} \label{calcul de e.}
e_{m+1}^{\mathcal{N}(e_{m+1})-1}\prod_{e\in\Lambda_n\backslash\Lambda_{m+1}} e^{\mathcal{N}(e)},
\end{equation} 
ce qui montre la majoration souhaitée de $e(K_1\dots K_n\vert F)$. 

Montrons maintenant la majoration du degré d'inertie. 
Le théorème \ref{inertiecompositum} devient maintenant une réécriture de la proposition précédente. 

Pour tout $e\in\N$ et tout $r\in\{1,\dots,n\}$, notons $\mathcal{N}_r(e)$ le nombre d'extensions $K_1/F,\dots K_r/F$ dont l'indice de ramification vaut $e$. 

 Fixons $r$ et $e\in\Lambda_r$. 
 Soit $q$ un nombre premier. 
Posons $A=\underset{j=1,...,r}{\pgcd}\left(\prod\limits_{\underset{i\neq j}{i=1}}^{r}e_i\right)$.
Alors $v_q(A)=\sum_{i=1}^r v_q(e_i)-\max\{v_q(e_1),\dots,v_q(e_r)\}$. 
Par définition de $\mathcal{N}_r(e)$, on obtient que \[\sum_{i=1}^r v_q(e_i)=\sum_{e\in\Lambda_r} \mathcal{N}_r(e)v_q(e).\] 
Un rapide calcul montre que $v_q(A)=\sum_{e\in\Lambda_r} (\mathcal{N}_r(e)-1)v_q(e)+a_r(q)$. 
Il en résulte donc que
 \begin{equation} \label{expression du E}
\underset{j=1,...,r}{\pgcd}\left(\prod_{\underset{i\neq j}{i=1}}^{r}e_i\right)=\prod\limits_{e\in\Lambda_r} e^{\mathcal{N}_r(e)-1}\prod\limits_{q\in\mathcal{P}} q^{a_r(q)}.
\end{equation}   

Supposons que $m\geq n-2$. 
De le proposition \ref{prop inutile}, on a $E=\underset{j=1,\dots,n}{\pgcd}\left(\prod_{\underset{i\neq j}{i=1}}^n e_i\right)$.
En prenant $r=n$ dans \eqref{expression du E}, on en déduit la majoration de $f(K_1\dots K_n\vert F)$ souhaitée.

Supposons maintenant que $m<n-2$. 
Dans ce cas, on a \[E = \underset{j=1,\dots,m+2}{\pgcd}\left(\prod_{\underset{i\neq j}{i=1}}^{m+2} e_i\right)\prod_{i=m+3}^n e_i.\]
En prenant cette fois-ci $r=m+2$ dans \eqref{expression du E}, on en déduit que 
\begin{equation*}
E=\prod\limits_{e\in\Lambda_{m+2}} e^{\mathcal{N}_{m+2}(e)-1}\prod\limits_{q\in\mathcal{P}} q^{a_{m+2}(q)}\prod_{i=m+3}^n e_i.
\end{equation*}

Comme $\prod\limits_{e\in\Lambda_{m+2}} e^{\mathcal{N}_{m+2}(e)}\prod_{i=m+3}^n e_i=\prod_{e\in\Lambda_n} e^{\mathcal{N}(e)}$, cela montre la majoration souhaitée de $f(K_1\dots K_n\vert F)$, ce qui achève la preuve du théorème \ref{inertiecompositum}.

\end{proof}

\begin{rqu}
Si on suppose en plus que les extensions $(K_i/F)_{i=1}^n$ soient galoisiennes, alors on peut remplacer "inférieur ou égal" par "divise" dans le théorème \ref{inertiecompositum} du fait que l'indice de ramification d'un compositum fini d'extensions galoisiennes de $F$ divise le produit des indices de ramification.
\end{rqu}

Soient $K_1/F, \dots, K_n/F$ comme dans le théorème \ref{inertiecompositum}. 
Pour un entier $i\leq n$, notons $e_i=e(K_i\vert F)$ et $f_i=f(K_i\vert F)$. 
Supposons que

\begin{enumerate} [i)]
\item $K_1\dots K_i/F$ et $K_{i+1}/F$ sont linéairement disjoints pour tout $i$;
\item $K_i/F$ est non sauvagement ramifié pour tout $i$; 
\item $e_i=e_j$ ou $\pgcd(e_i,e_j)=1$ pour tous $i,j$. 
\end{enumerate}
 Alors dans le théorème \ref{inertiecompositum}, notre majoration du degré d'inertie est une égalité.  
En effet, il s'ensuit de la condition $i)$ que \[ [K_1\dots K_n : F]=\prod_{i=1}^n [K_i : F]=\prod_{i=1}^n e_if_i.\]

De plus, les extensions $K_i/F$ sont non sauvagement ramifiées. 
 Le théorème \ref{Abhyankar} permet donc d'en déduire que 
\begin{equation} \label{exempledfdf}
e(K_1\dots K_n\vert F)=\ppcm(e_1,\dots,e_n).
\end{equation} 
 
De la condition $iii)$, on obtient que $a_n(q)=0$ pour tout premier $q$. 
De plus, la condition $iii)$ et \eqref{exempledfdf} permettent d'en déduire que \[e(K_1\dots K_n\vert F)=\prod_{e\in\Lambda_n} e.\]     

Comme $[K_1\dots K_n : F]=e(K_1\dots K_n\vert F)f(K_1\dots K_n\vert F)$, on en déduit que \[f(K_1\dots K_n \vert F)=\prod_{e\in\Lambda_n} e^{\mathcal{N}(e)-1}\prod_{i=1}^n f_i.\]
Le théorème \ref{inertiecompositum} permet d'en déduire que \[ \prod_{i=1}^n f_i \leq \ppcm(f_1,\dots,f_n).\]
Cela montre donc que $\prod_{i=1}^n f_i=\ppcm(f_1,\dots,f_n)$.
Ainsi, la majoration du degré d'inertie du théorème \ref{inertiecompositum} est, dans ce cas, une égalité. 
 
Nous terminons cet appendice avec deux exemples où l'on compare les bornes obtenues en utilisant d'un côté le corollaire \ref{corollaire immédiat du thm de Krasner} et de l'autre le théorème \ref{inertiecompositum}. 
L'un contiendra de la ramification sauvage et l'autre non.

\begin{ex}
\rm{Prenons $p=11$ et $\{ K_1,\dots, K_n\}$ l'ensemble des extensions de $\Q_{11}$ de degré plus petit que $10$ (c'est bien un ensemble fini d'après le théorème \ref{Krasner}). 
Notons $K$ le compositum des $K_n$. 
Pour $i\in\{2,...,10\}$, le théorème \ref{Krasner} montre que $\mathcal{N}_{\Q_{11},i}=\sum\limits_{d\mid i} d$. 
Ainsi, la majoration donnée par le corollaire \ref{corollaire immédiat du thm de Krasner} montre que
\begin{equation*}
f(K\vert\Q_{11})\leq\prod\limits_{i=1}^{10} i^{\sum\limits_{d\mid i} d}\leq 1,9.10^{71}.
\end{equation*}

Calculons, avec le théorème \ref{inertiecompositum}, une majoration plus précise de $f(K\vert \Q_{11})$. 
Comme il n'y a pas de ramification sauvage, on est dans le cas où $m=n>m-2$. 
Comme il existe une unique extension non ramifiée de $\Q_{11}$ de degré $f$, que l'on note $\Q_{11}\{f\}$, il s'ensuit que le nombre d'extensions de $\Q_{11}$ de degré $ef$ et de degré d'inertie $f$ est, d'après le théorème \ref{Krasner}, $\mathcal{N}_{\Q_{11}\{f\}, e}^{(r)}=e$.
Dans notre situation, $f\leq 10e^{-1}$ et donc, le nombre d'extensions dont le degré de ramification vaut $e$ est $e\lfloor 10/e\rfloor$.
Rappelons que pour tout premier $q$, 
\begin{equation*}
a_n(q)=-\max\{v_q(1),\dots,v_q(10)\}+\sum\limits_{e=1}^{10} v_q(e).
\end{equation*}

 Ainsi, $a_n(2)=5,\;a_n(3)=2,\; a_n(5)=1$ et $a_n(q)=0$ pour les autres valeurs de $q$. Il en résulte donc que
\begin{equation*}
f(K\vert\Q_{11})\leq \ppcm(1,\dots,10)\;2^5\; 3^2\; 5 \prod\limits_{e=1}^{10} e^{e\;\lfloor\frac{10}{e}\rfloor-1}\leq 3,3.10^{56}.
\end{equation*}
}

\end{ex}

Traitons le second exemple.

\begin{ex}
\rm{Prenons $p=5$ et $\{K_1,\dots, K_n\}$ l'ensemble des extensions de $\Q_{5}$ de degré plus petit que $10$. Notons $K$ le compositum des $K_n$. Le théorème \ref{Krasner} montre que $\mathcal{N}_{\Q_5,5}=106$ et $\mathcal{N}_{\Q_5,10}=1818$. Ainsi, en utilisant la majoration donnée par le corollaire \ref{corollaire immédiat du thm de Krasner}, on en déduit que
 \begin{equation*}
f(K\vert\Q_{5})\leq 5^{106}\;10^{1818}\; \prod_{\underset{i\neq 5}{i=1}}^9 i^{\sum_{d\mid i} d}\leq 1,5.10^{1941}.
\end{equation*}

Calculons maintenant un majorant de $f(K\vert \Q_5)$ à l'aide du théorème \ref{inertiecompositum}. 
Notons $m$ le nombre d'extensions $K_j/\Q_5$ non sauvagement ramifiées.
On a $m<n-2$.
D'après l'exemple précédent, $\mathcal{N}(e)=e \lfloor 10/e\rfloor$ si $e\notin\{ 5,10\}$. 
Pour $e=5$, on a $f=1$ ou $f=2$.
Ainsi, d'après \eqref{majoration du nombre N(e).}, on a $\mathcal{N}(5)\leq 105+605=710$.
De même, si $e=10$, alors $f=1$ et on en déduit que $\mathcal{N}(10)\leq 1210$. 
On en déduit donc (en prenant $e_{m+1}=e_{m+2}=10$) que $a_{m+2}(2)=8-3=5,\; a_{m+2}(3)=4-2=2$ et $a_{m+2}(q)=0$ pour les autres valeurs de $q$. On en conclut donc que: 
\begin{equation*}
f(K \vert \Q_5)\leq \ppcm(1,\dots,10)\;(10!)^{-1}\;5\;2^5\;3^2\; 5^{710}\;10^{1210}\prod\limits_{\underset{e\neq 5}{e=1}}^9 e^{e\;\lfloor\frac{10}{e}\rfloor}\leq 6,2.10^{1745}.
\end{equation*}
}
\end{ex}

\bibliographystyle{plain}

\end{document}